\documentclass[11pt,reqno]{amsart}

\usepackage[utf8]{inputenc} 
\usepackage[T1]{fontenc}
\usepackage{url}
\usepackage{tcolorbox}
\usepackage{cite}
\usepackage{amsmath} 
                             
\allowdisplaybreaks

\usepackage{soul,times,amsfonts,amssymb,amsthm,bbm,mathrsfs,xcolor,graphicx,enumitem,booktabs,euscript,mathtools,nicefrac,cite,float,graphicx}
\usepackage[font=footnotesize]{subcaption}
\usepackage[normalem]{ulem}
\usepackage{cuted}
\usepackage{flushend}
\usepackage[
colorlinks,bookmarks=true,  linkcolor={red!50!black},
    citecolor={blue!50!black},
    urlcolor={blue!80!black}, pdfencoding=auto, psdextra]{hyperref}
\makeatletter
\pretocmd{\subsubsection}{\vspace{6pt}}{}{}
\makeatother

\newcommand\redout{\bgroup\markoverwith{\color{red}\rule[0.5ex]{2pt}{0.8pt}}\ULon}

\newtheorem{theorem}{Theorem}[section]
\newtheorem{lemma}[theorem]{Lemma}

\newtheorem{proposition}[theorem]{Proposition}

\theoremstyle{definition}
\newtheorem{definition}{Definition}[section]
\theoremstyle{remark}
\newtheorem{remark}{Remark}
\newtheorem{example}{Example}
\allowdisplaybreaks

\usepackage[left=1.2in,right=1.2in,top=1.4in,bottom=1.4in]{geometry}

\newcommand\nc\newcommand
\nc\bfa{{\boldsymbol a}}\nc\bfA{{\boldsymbol A}}\nc\sA{{\EuScript A}}\nc\cA{{\mathcal A}}
\nc\bfb{{\boldsymbol b}}\nc\bfB{{\boldsymbol B}}\nc\cB{{\mathcal B}}\nc\sB{{\EuScript B}}
\nc\bfc{{\boldsymbol c}}\nc\bfC{{\boldsymbol C}}\nc\cC{{\mathscr C}}
\nc\bfd{{\boldsymbol d}}\nc\bfD{{\boldsymbol D}}\nc\cD{{\mathscr D}}
\nc\bfe{{\boldsymbol e}}\nc\bfE{{\boldsymbol E}}\nc\cE{{\mathcal E}}\nc\sE{{\mathscr E}}
\nc\bff{{\boldsymbol f}}\nc\bfF{{\boldsymbol F}}\nc\cF{{\mathscr F}}
\nc\bfg{{\boldsymbol g}}\nc\bfG{{\boldsymbol G}}\nc\cG{{\EuScript G}}
\nc\bfh{{\boldsymbol h}}\nc\bfH{{\boldsymbol H}}\nc\cH{{\mathcal H}}\nc\sH{{\mathscr H}}
\nc\bfi{{\boldsymbol i}}\nc\bfI{{\boldsymbol I}}\nc\cI{{\EuScript I}}\nc\sI{{\mathscr I}}
\nc\bfj{{\boldsymbol j}}\nc\bfJ{{\boldsymbol J}}\nc\cJ{{\EuScript J}}
\nc\bfk{{\boldsymbol k}}\nc\bfK{{\boldsymbol K}}\nc\cK{{\EuScript K}}
\nc\bfl{{\boldsymbol l}}\nc\bfL{{\boldsymbol L}}\nc\cL{{\EuScript L}}
\nc\bfm{{\boldsymbol m}}\nc\bfM{{\boldsymbol M}}\nc\cM{{\EuScript M}}\nc\sM{{\mathscr M}}
\nc\bfn{{\boldsymbol n}}\nc\bfN{{\boldsymbol N}}\nc\cN{{\EuScript N}}
\nc\bfo{{\boldsymbol o}}\nc\bfO{{\boldsymbol O}}\nc\cO{{\EuScript O}}\nc\sO{{\mathscr O}}
\nc\bfp{{\boldsymbol p}}\nc\bfP{{\boldsymbol P}}\nc\cP{{\EuScript P}}\nc\sP{{\mathscr P}}
\nc\bfq{{\boldsymbol q}}\nc\bfQ{{\boldsymbol Q}}\nc\cQ{{\mathcal Q}}
\nc\bfr{{\boldsymbol r}}\nc\bfR{{\boldsymbol R}}\nc\cR{{\EuScript R}}
\nc\bfs{{\boldsymbol s}}\nc\bfS{{\boldsymbol S}}\nc\cS{{\EuScript S}}
\nc\bft{{\boldsymbol t}}\nc\bfT{{\boldsymbol T}}\nc\cT{{\EuScript T}}\nc\sT{{\mathscr T}}
\nc\bfu{{\boldsymbol u}}\nc\bfU{{\boldsymbol U}}\nc\cU{{\EuScript U}}
\nc\bfv{{\boldsymbol v}}\nc\bfV{{\boldsymbol V}}\nc\cV{{\mathscr V}}
\nc\bfw{{\boldsymbol w}}\nc\bfW{{\boldsymbol W}}\nc\cW{{\mathscr W}}
\nc\bfx{{\boldsymbol x}}\nc\bfX{{\boldsymbol X}}\nc\cX{{\EuScript X}}
\nc\bfy{{\boldsymbol y}}\nc\bfY{{\boldsymbol Y}}\nc\cY{{\mathscr Y}}
\nc\bfz{{\boldsymbol z}}\nc\bfZ{{\boldsymbol Z}}\nc\cZ{{\EuScript Z}}
\nc{\1}{{\mathbbm{1}}}
\nc\cross{{\mathsf{C}}}

\nc{\remove}[1]{}

\newcommand{\bfit}{\bfseries\itshape}

\DeclareSymbolFont{bbold}{U}{bbold}{m}{n}
\DeclareSymbolFontAlphabet{\mathbbold}{bbold}

\DeclareMathOperator{\supp}{supp}

\newcommand{\R}{{\mathbb R}}

\newcommand{\Z}{{\mathbb Z}}
\newcommand{\E}{{\mathbb E}}
\nc\torus{{\mathbb{T}}}
\nc\reals{{\mathbb R}}
\nc{\ff}{{\mathbb F}}
\nc{\PP}{{\mathbb P}}
\nc{\complex}{{\mathbb C}}

\nc\inftyeq{\overset{\infty}{=}}

\usepackage{tikz}
\usetikzlibrary{positioning}
\usetikzlibrary{calc}
\usetikzlibrary{decorations.pathreplacing}

\definecolor{newred}{HTML}{ff382e}
\definecolor{newgreen}{HTML}{549641}
\definecolor{newblue}{HTML}{4c4cfc}
\definecolor{neworange}{HTML}{c98702}

\newcommand{\tvd}[1]{\Vert {#1}\Vert_{\mathnormal{TV}}}

\nc{\new}[1]{\textcolor{blue}{#1}}

\title[]{The maximal hard-core model as a recoverable system: \\ Gibbs measures and phase coexistence}

\author[]{Geyang Wang$^1$}

\author[]{Alexander Barg$^1$}

\author[]{Navin Kashyap$^2$}

\begin{document}

%\footnotetext[1]{Emails: \{wanggy,abarg\}@umd.edu, nkashyap@iisc.ac.in}
\footnotetext[1]{ISR and Department of ECE, University of Maryland, College Park, MD 20742, USA. Emails:\{wanggy,abarg\}@umd.edu.}

\footnotetext[2]{Department of Electrical Communication Engineering, Indian Institute of Science, Bangalore 560012, India. Email {nkashyap@iisc.ac.in}.}

\begin{abstract} Recoverable systems provide coarse models of data storage on the two-dimensional square lattice $\Z^2$. In a recoverable system, each lattice site stores one bit of data, and the overall configuration of bits on the lattice is such that each site can reconstruct the value from its neighbors' values according to a specified local rule. 
We initiate a study of such systems through the lens of statistical models for interacting bits on the $\Z^2$ lattice. We consider two statistical models: (a) an interaction model with a temperature parameter that encourages configurations that respect the local recovery rule, and (b) a hard-core model with an activity parameter, supported only on the admissible configurations of the recoverable system.

In this paper, we specifically study the recoverable system in which the admissible configurations are precisely those whose supports form maximal independent sets in the lattice. For the interaction model, we establish uniqueness of the Gibbs measure at high temperature and derive bounds on the entropy in the zero- and low-temperature regimes. For the hard-core model of maximal independent sets, we show phase coexistence at high activity. Unlike the standard hard-core model, however, the maximal version admits nontrivial ground states even at low activity, and we manage to classify them explicitly. We further verify the Peierls condition for appropriately defined contours. Combined with the Pirogov–Sinai theory, this shows that each ground state gives rise to an extremal Gibbs measure, proving phase coexistence at low activity.
\end{abstract}

\maketitle

\thispagestyle{empty}
\vspace*{-.1in}
{\footnotesize{\tableofcontents}}

\section{Introduction}

A recoverable system is a collection of spin configurations on a lattice with the property that for every configuration within the collection, the spin at any lattice site is uniquely determined by the spins at its neighbors. As we explain below, recoverable systems originate from the concept of local recoverability in coding theory. In this paper, we exclusively work with a recoverable system $\cX_0$, defined further below, on the two-dimensional integer lattice $\Z^2$ with spins taking values in the set $\{0,1\}$. We will view this recoverable system as a nearest-neighbor interaction model, where the distribution of the spin at a particular site is influenced by the other spins in the configuration only through the spins at the nearest neighbors. Our primary aim is to introduce problems of a statistical physics flavor for this model, that are motivated in part by considerations that arise in the study of recoverable systems.

Recoverable systems extend the concept of locally recoverable codes, which were introduced in coding theory over a decade ago \cite{gopalan2011locality} and by now form an established area of research \cite{Ramkumar2022}. The concept of local reconstruction is motivated by the need to reduce communication required for the recovery of a dysfunctional storage element by addressing a subset of the nodes
used to store the encoded block of data. If every coordinate of the codeword is a function of a small number of other coordinates, this supports the local recovery, contributing to the efficient functioning of the storage system. A version of this problem with graphical constraints, introduced in \cite{Mazumdar2015,Shanmugam2014}, assumes that the coordinates of the code are in a one-to-one correspondence with the 
vertices of a finite graph, and that the local recovery of a vertex relies on the values assigned to its neighbors in the graph. 
Apart from being of independent interest, this problem turned out to be linked to several other combinatorial problems, notably, the {min-rank} problem for graphs (see \cite{barg2022high} for an overview of the literature). Some recent results on codes with locality on triangle-free graphs appear in \cite{barg2024storage,huang2023construction}. In a further development, the problem of local recovery was extended to infinite graphs such as graphs on $\Z^d, d\ge 1$  \cite{elishcoRecoverableSystems2022},\cite{barg2024LinesGrids}, where it was studied relying on methods from symbolic dynamics and constrained systems.

We introduce some terminology needed to formally define a recoverable system on the $\Z^2$ lattice, which we view as an infinite graph. The vertices of this graph are the sites $i \in \Z^2$, and two vertices are connected by an edge if one of their two coordinates coincides and the other differs by one. Two vertices $i,j$ are called adjacent or \emph{neighbors}, denoted $i \sim j$, if they are connected by an edge, and two edges are called adjacent if they share a common vertex.  Any assignment of bits to the points (sites) of the grid is called a {\em configuration}, denoted $\omega$. Let $\Omega := \{0,1\}^{\Z^2}$ be the set of all configurations. A subset $\cX\subset\Omega$ is said to form a {\em recoverable system} if for every $\omega\in \cX$ and every site $i\in\Z^2$, the value of its bit $\omega_i$ is uniquely determined by the values of its $4$ neighbors in the lattice, in other words, there is a function $f:\{0,1\}^{4}\to\{0,1\}$ such that $f(\{\omega_j,i\sim j\})=\omega_i$ for all $i\in \Z^2$. If, as this notation suggests, $f$ does not depend on $i$, the system $\cX$ is called shift-invariant. 

Our focus in this paper is on the recoverable system $\cX_0$ formed of all configurations $\omega \in \Omega$ that obey the following recovery rule: a site is 1 if and only if all of its 4 neighbors are 0. If at least one of the 4 neighbors of $i$ is 1, then $\omega_i=0$. In other words, 
  \begin{equation}  \label{eq:recovery}
  \omega_{(0,0)}=1 \;\Longleftrightarrow\;
\omega_{(0,1)}=\omega_{(1,0)}=\omega_{(0,-1)}=\omega_{(-1,0)}=0,
  \end{equation}
   extended to the other sites by translation. 
    
There are two reasons to study this system. First, having information storage in mind, we are interested in ``large-size'' systems (systems with positive topological entropy). To argue that the system $\cX_0$ has positive entropy,
we recall the connection between recoverable systems and {constrained systems}, discussed in detail in \cite{elishcoRecoverableSystems2022} (in the case of one dimension). Specifically, a bi-infinite binary sequence in one dimension is said to satisfy a $(d,k)$-{\em constraint} if any pair of successive ones is separated by at least $d$ consecutive zeros, and every run of consecutive zeros has length at most $k$. A two-dimensional (2D) constrained system $S_{(d_1,k_1),(d_2,k_2)}$ is the collection of all configurations, i.e., assignments of bits to the sites of $\Z^2$, such that in each configuration, every row satisfies the $(d_1,k_1)$-constraint and every column satisfies the $(d_2,k_2)$-constraint. Now consider the system $S_{(1,2),(1,\infty)}$: for any $\omega$ in this system, no two ones are adjacent either horizontally or vertically, every zero is adjacent to a 1 in its row, so every such configuration satisfies \eqref{eq:recovery}. We conclude that $S_{(1,2),(1,\infty)}\subset \cX_0$. By a result of \cite{katozegger2000ISIT}, this system has positive entropy, and thus, $H(\cX_0)>0$. 
Viewing the system $\cX_0$ as a {\em subshift of finite type (SFT)}, one can also prove that it has positive entropy. This link is discussed in more detail in Appendix~\ref{app: SFT}.

Another motivation comes from the following connection to {independent sets} in the $\Z^2$ grid graph.
Given a configuration $\omega$, we say that a site $i$ is occupied or unoccupied in $\omega$ depending on whether $\omega_i=1$ or $0$.
A configuration forms an {independent set} in $\Z^2$ if no two occupied sites are adjacent. An independent set is a {\em maximal independent set}, or MIS, if no unoccupied site can be made occupied without breaking the independence condition.
Let $\sI(\Z^2)$ ($\sM(\Z^2)$) denote the set of independent sets (resp., maximal independent sets) in $\Z^2$.

\begin{lemma}\label{lemma:MIS} 
A configuration $\omega\in \Omega$ forms an MIS if and only if $\omega\in\cX_0$ (see Fig.~\ref{fig:MIS} for an illustration of this equivalence).
\end{lemma}
\begin{proof}
Let $\omega\in\cX_0$ and let $\supp (\omega) := \{i \in \Z^2, \omega_i = 1\}$. From \eqref{eq:recovery}, $i$ is occupied 
only if its 4 neighbors are not, and so $\supp(\omega)\in \sI(\Z^2)$. Moreover, if $\supp (\omega)$ is not maximal, then
there is an unoccupied site $i$ all of whose neighbors also are unoccupied, but this contradicts \eqref{eq:recovery}.
Thus, $\supp(\omega)\in \sM(\Z^2)$ for all $\omega\in\cX_0$. Conversely, if $I\in \sM(\Z^2)$, then the configuration $\omega$ with $\supp(\omega)=I$ satisfies 
the recovery rule \eqref{eq:recovery}.
\end{proof}

\subsection{Two statistical models arising from $\cX_0$}

Drawing upon the two distinct motivations we have given for studying the recoverable system $\cX_0$, we consider in this paper two different statistical models on configurations on $\Z^2$. The first is an interaction model defined via a potential $\Phi$ with parameter $\beta > 0$ that encourages configurations in $\cX_0$. This model is rooted in the physics of information storage, and is inspired in part by a similar model studied by Goldfeld et al.\ \cite{goldfeld2021Ising}. It is an attempt to model the physical degradation of a storage system that stores information in a configuration from $\cX_0$. The inverse temperature parameter $\beta$ determines the rate at which the system degrades.  Under ideal conditions, i.e., at zero temperature ($\beta = \infty$), enforcement of the recovery rule \eqref{eq:recovery} suffices to keep the system in the initial configuration belonging to $\cX_0$. But as the temperature increases, the recovery rule is unable to keep up with errors caused by physical effects, and the system starts to deviate from ideal behavior, causing the system to move away from the initial configuration. We model the system evolution as a Glauber dynamics running on a (possibly large) finite region $\Lambda \Subset \Z^2$ (here, $\Subset$ means a finite subset) and controlled by a Hamiltonian $\sH_{\Lambda;\Phi}$ that is parameterized, via $\Phi$, by $\beta$. The Glauber chain starts at some configuration in $\cX_0$ restricted to $\Lambda$. At each time step, it picks a site $i \in \Lambda$ uniformly at random and updates its value by sampling from the conditional Gibbs distribution at site $i$, given the values at sites other than $i$. A problem of interest then is to determine the mixing time of the Glauber chain as a function of the parameter $\beta$.

Further, the stationary distribution of the Glauber chain is the Gibbs distribution on $\Lambda$ defined by the Hamiltonian $\sH_{\Lambda;\Phi}$, which, we emphasize again, is parameterized by $\beta$. It is naturally of interest to determine how much a configuration drawn from this Gibbs distribution deviates from a configuration in $\cX_0$. We do this by estimating the site error probability, i.e., probability of a site in $\Lambda$ not conforming to the recovery rule \eqref{eq:recovery}. It is also instructive to compute/estimate/bound the entropy of the Gibbs distribution at positive temperature ($\beta < \infty$) as well as at zero temperature. These entropies are useful in quantifying the amount of information that can be stored in configurations drawn from the Gibbs distribution. We expect that configurations at positive temperature will have a higher entropy, so can store more information, but will typically have a certain proportion of sites not satisfying the recovery rule \eqref{eq:recovery}. Moreover, it is important to develop an understanding of what typical configurations at positive temperature look like, for which it is necessary to describe the statistical behavior of Gibbs measures in this regime.

From the perspective of information storage, the zero-temperature limit is of particular interest, since in this regime the Gibbs measure is supported on $\cX_0$, which may be viewed as a measure of maximal entropy for the associated SFT. A fundamental problem then is to understand the non-uniqueness of measures of maximal entropy~\cite{burton1995new,burton1994nonuniqueness}, i.e., the phase coexistence problem at zero temperature. Although phase coexistence for SFTs has been extensively studied using Dobrushin-type uniqueness methods, these techniques do not appear sufficient to resolve the present model. Our application of the Dobrushin uniqueness condition for sufficiently small inverse temperature $\beta$ should therefore be viewed as a step toward understanding this limiting regime.

The second statistical model we consider is a hard-core model with an activity parameter $\lambda > 0$, that is supported only on configurations in $\cX_0$, i.e., MISs. We refer to the latter model as the \emph{maximal hard-core model}; the term \emph{hard-core model}, without further qualification, will invariably refer to the usual hard-core model on independent sets (not necessarily maximal). 
For the hard-core model, a significant amount of attention has been paid to understanding precisely the regimes in which there is uniqueness of Gibbs measure and when there is phase co-existence (see, e.g., \cite{berg1994percolation, scott2005repulsive, blanca2019phase, Jenssen_2024, cannon2024pirogov, mazel2019high, mazel2025high}). These problems are naturally of interest for the maximal hard-core model as well. A major difference between the hard-core and maximal hard-core models is the existence of phase transitions even at low activity $\lambda$ in the latter model, which we show in this work. This raises the interesting open problem of whether there is \emph{any} range of $\lambda$ for which there is uniqueness of Gibbs measure in the maximal hard-core model.
The case $\lambda = 1$ is of particular interest, since the Gibbs measure of the maximal hard-core model is an entropy maximizing measure on $\cX_0$. Consequently, our quantitative bounds quantify the gap between the currently accessible phase coexistence regime and the entropy-maximization point $\lambda=1.$

There is also an information-theoretic motivation for developing a fine-grained understanding of Gibbs measures for the maximal hard-core model: We would like to gauge the ability of the recoverable system $\cX_0$ to retain information when acted upon by a noise process. Concretely, consider a noisy channel that takes as input a configuration in $\cX_0$, and flips the bit at each site with probability $p$ independently of other sites. For such a channel, it would be of interest to determine the mutual information rate between channel input and output, when the input to the channel is a configuration drawn from a measure $\mu$ on $\cX_0$. The idea then would be to maximize the mutual information rate over measures $\mu$ belonging to some parameterized class, such as the class of Gibbs measures parameterized by $\lambda$. This fits neatly into an established line of work on the capacity of channels with constrained inputs \cite{zehavi1988runlength,han2009asymptotics,sabato2012gbp,molkaraie2013montecarlo,sabag2016bec,li2018asymptotics,sabag2018bibo}.

The task of information storage and recovery also provides an added motivation for the problem of finding the exact separation points between the regions of uniqueness and phase coexistence in terms of $\lambda$ since the behavior of the system changes depending on the region. In the uniqueness regime the system exhibits fast correlation decay, 
implying that it can store large volumes of information. In the coexistence regime, depending on the boundary
condition, Gibbs measures can be nonergodic and can support long-range order, directly impacting the storage
capacity.

As noted above, the hard-core model has been extensively studied in probabilistic combinatorics and statistical physics, including the problem of locating the value(s) of critical activity. In contrast, the maximal hard-core model seems to have received far less attention. One notable prior work on this model is that of Dall'Asta et al.\ \cite{dallasta2009}, where the authors study MISs in random graphs relying on the cavity method and Monte-Carlo simulations.

\subsection{Overview of the paper and summary of results}

The paper is formed of two parts, of which the first (Sections \ref{sec:high-temp}--\ref{sec:entropy}) is largely devoted to the interaction model at positive temperature, and the second part (Sections~\ref{sec:external_field}--\ref{sec:low avtivity}) addresses phase coexistence in the maximal hard-core model. In the first part, we show the absence of phase transitions at high temperature, relying on Dobrushin's uniqueness criterion \cite[Theorem~6.31]{friedliLattice2018}. In this part, we also estimate the entropy of the system at both positive and zero temperature. Calculations for the zero-temperature regime are based on the transfer matrix method, often employed in the study of constrained systems \cite{forchhammer2000}. At positive temperature, not all the sites in a typical configuration will follow the recovery rule, and we estimate the error probability of a site as a function of the temperature.

Our main results are related to the maximal hard-core model, which addresses the study of phase
transitions controlled by the activity parameter $\lambda$ of the model. This version of the problem can also be understood as a study of systems that follow the recovery rule \eqref{eq:recovery} exactly, but are typically formed of sparser or denser MIS configurations according as $\lambda$ is low or high. Here, we prove the presence of phase transitions at high activity. We build on a recent result for the more well-studied hard-core model on $\sI(\Z^2)$, \cite{blanca2019phase}, which established phase transitions for high $\lambda$. The main idea in this part is a contour elimination procedure, which for the hard-core case originates with Dobrushin's paper \cite{dobrushin1968problem}. 

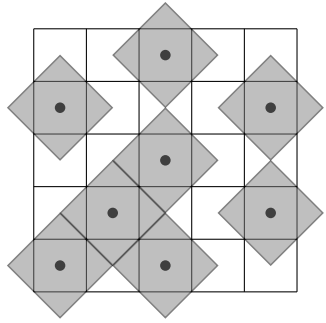
\begin{figure}[t]
    \centering
   {
        \begin{tikzpicture}[scale=0.70]
        \foreach \x in {0,1,2,3,4,5} {
            \draw (\x, 0) -- (\x, 5); 
            \draw (0, \x) -- (5, \x); 
        }
        \foreach \pos in {(0.5,0.5), (0.5,3.5), (1.5,1.5), (2.5,0.5), (2.5,4.5), (2.5,2.5), (4.5,1.5), (4.5,3.5)} {
            \fill \pos circle(0.1);
        }
             \filldraw[fill=gray, opacity=0.5, draw=black, line width=0.2mm] [rotate around={45:(0.5,0.5)}] (-0.2,-0.2) rectangle (1.2,1.2);
        \filldraw[fill=gray, opacity=0.5, draw=black, line width=0.2mm] [rotate around={45:(1.5,1.5)}] (.8,.8) rectangle (2.2,2.2);
        \filldraw[fill=gray, opacity=0.5, draw=black, line width=0.2mm] [rotate around={45:(0.5,3.5)}] (-.2,2.8) rectangle (1.2,4.2);
        \filldraw[fill=gray, opacity=0.5, draw=black, line width=0.2mm] [rotate around={45:(2.5,0.5)}] (1.8,-.2) rectangle (3.2,1.2);
   \filldraw[fill=gray, opacity=0.5, draw=black, line width=0.2mm] [rotate around={45:(2.5,2.5)}] (1.8,1.8) rectangle (3.2,3.2);
   \filldraw[fill=gray, opacity=0.5, draw=black, line width=0.2mm] [rotate around={45:(4.5,3.5)}] (3.8,2.8) rectangle (5.2,4.2);
   \filldraw[fill=gray, opacity=0.5, draw=black, line width=0.2mm] [rotate around={45:(2.5,4.5)}] (1.8,3.8) rectangle (3.2,5.2);
    \filldraw[fill=gray, opacity=0.5, draw=black, line width=0.2mm] [rotate around={45:(4.5,1.5)}] (3.8,0.8) rectangle (5.2,2.2);  
    \end{tikzpicture}
    }
    \caption{The blank/dotted squares in the grid correspond to 0's and 1's, respectively. The 1's not on the boundary are surrounded by 4 0's, and every 0 is adjacent to at least one 1; cf.~\eqref{eq:recovery}. The set of 1's forms an MIS in the grid graph, or a {\em maximal hardcore configuration} in the lattice, shown by the nonoverlapping gray squares.}\label{fig:MIS}

\end{figure}

As already remarked, a new feature of the maximal hard-core model arises from observing that, unlike the hard-core model on $\sI(\Z^2)$, for the maximal case there are nontrivial ground states even at low $\lambda$. Such states are relatively easy to enumerate based on elementary geometric considerations, which we do in Theorem~\ref{prop: density}. As it turns out, there are 10 sparsest periodic MIS configurations, which are all related by $\Z^2$ isometries, and each of them
has density $1/5$ as a subset of $\Z^2$ (in the case of dense sets, maximality does not play a role, so there are two configurations of density 1/2, namely, the odd and even patterns). For these states, we further verify the Peierls condition \cite[Sec.~7.2]{friedliLattice2018}, \cite{zahradnik1984alternate,zahradnik1998short} for the occupancy versus the contour length, see Theorem~\ref{thm: Peierls}. The tool used in the proof is based on Delaunay triangulations of the ground states as subsets of $\R^2$. Enumeration of periodic ground states and the contour length condition enable the application of the classic {\em Pirogov-Sinai argument}, which implies that these states 
give rise to distinct periodic extremal Gibbs measures.

\section{Preliminaries} \label{sec:prelims}
In this section, we introduce elements of notation and terminology used in the paper. We make no attempt at completeness; see, for instance, \cite[Ch.~6]{friedliLattice2018}, \cite[Ch.~2, 4]{Georgii2011} for details.

For a region $\Lambda \Subset \Z^2$ (here $\Subset$ means a {\em finite subset}) and a configuration $\omega = {(\omega_i)}_{i \in \Z^2}$, let $\omega_\Lambda = {(\omega_i)}_{i \in \Lambda}$ be its restriction to $\Lambda$,  and let $\Pi_\Lambda: \Omega \to \Omega_\Lambda$ be the projection map defined by $\Pi_\Lambda(\omega) = \omega_\Lambda $. Let $\Omega_\Lambda:=\{\omega_\Lambda|\omega\in \Omega\}$. 

Gibbs measures can be defined on the space $(\Omega,\cF)$, where $\cF$ is the cylinder $\sigma$-algebra on $\Omega$. 
To define it, we start with a finite region $\Lambda \Subset \Z^2$ and an event $A \in \sP(\Omega_\Lambda)$, where $\sP(\Omega_\Lambda)$ is the set of all possible subsets of $\Omega_\Lambda$. A {\em cylinder} with base $\Lambda$
is defined as  $\Pi^{-1}_{\Lambda}(A) := \{\omega \in \Omega \mid \omega_\Lambda \in A\}$. 
Let
\[
\cC(\Lambda) := \{\Pi^{-1}_\Lambda (A) : A \in \sP(\Omega_\Lambda)\}
\]
be the collection of all cylinders with base $\Lambda$. We then define, for any $S \subset \Z^2$, the $\sigma$-algebra 
\[
\cF_S := \sigma \Big (\bigcup_{\Lambda \Subset S} \cC (\Lambda) \Big ).
\]
Finally, $\cF:=\cF_{\Z^2}$.
Next, let us define a {\em potential}, i.e., a collection of functions $\Phi = \{\Phi_B\}_{B \Subset \Z^2}$ such that for each finite $B \Subset \Z^2$, $\Phi_B : \Omega \mapsto \R$ is $\cF_B$-measurable.
Below, we will only consider $B$ to be the cross-shaped set formed of the 5 sites involved in the definition of the recovery rule \eqref{eq:recovery}, so we do not give the most general definitions. 
The {\em Hamiltonian} (energy) in $\Lambda \Subset \Z^2$ associated with the potential $\Phi$ is defined by 
\begin{equation*}\label{eq:Hamiltonian}
    \sH_{\Lambda;\Phi}(\omega) := \sum_{\substack{B \Subset \Z^2 \\ B \cap \Lambda \ne \emptyset }} \Phi_B(\omega), \quad \forall \omega \in \Omega.
\end{equation*}

Further, let us fix a configuration $\omega \in \Omega$, which we will refer to as a \emph{boundary condition}. 
%In particular, if $\omega=0^{\Z^2}$, we will call it a \emph{zero boundary condition}. 
For $\Lambda\Subset \Z^2$, let $\partial\Lambda:=\{i \in \Lambda^c \mid i \sim j, j \in \Lambda\}$ be its
(immediate outer) boundary. Denote the complement of $\Lambda$ in $\Z^2$ by $\Lambda^c$, and let $\omega_{\Lambda^c}:={(\omega_i)}_{i\in\Lambda^c}$. 
For $\Lambda\Subset \Z^2$, let $\Omega_\Lambda^\omega:=\{\tau\in \Omega\mid \tau_{\Lambda^c}=\omega_{\Lambda^c}\}$ be the set of configurations that agree with $\omega$ on $\Lambda^c$. 
Define a probability measure $\pi_\Lambda(\cdot\mid\omega)$ on $(\Omega,\cF)$ with support  $\Omega_\Lambda^\omega$, which assigns to each configuration $\tau_{\Lambda}\omega_{\Lambda^c} \in \Omega_{\Lambda}^{\omega}$ the probability 
   \begin{equation}\label{eq:specification}
    \pi_{\Lambda}(\tau_\Lambda \mid \omega) := \frac{1}{\bfZ^\omega_{\Lambda}}  e^{-\sH_{\Lambda; \Phi}(\tau_\Lambda \omega_{\Lambda^c})},
   \end{equation}
where $\bfZ^\omega_{\Lambda} := \sum_{\tau_\Lambda \in \Omega_\Lambda} e^{-\sH_{\Lambda; \Phi}(\tau_\Lambda \omega_{\Lambda^c})}$ is the partition function. Note that $\pi_\Lambda(\Omega_\Lambda^\omega \mid \omega) = 1$. 

For a given $\Lambda \Subset \Z^2$, as we let $\omega$ vary over the configurations in $\Omega$, we get a collection of probability measures $\{\pi_\Lambda(\cdot \mid \omega)\}_{\omega \in \Omega}$, which we collectively refer to as the \emph{probability kernel} $\pi_{\Lambda}$. The collection $\pi = \{\pi_\Lambda\}_{\Lambda \Subset \Z^2}$ of probability kernels is called a \emph{Gibbsian specification}. A measure $\mu$ on $(\Omega,\cF)$ is said to be \emph{compatible} with $\pi$ if 
$\mu=\mu\pi_\Lambda$ for all $\Lambda\Subset\Z^2$, where $\mu \pi_\Lambda(A) := \int \pi_\Lambda(A \mid \omega) \mu(d\omega), \ A \in \cF.$ Such a $\mu$ is called a {\em Gibbs measure}. From the definition of the conditional probability $\mu(A \mid \cF_{\Lambda^c})(\omega)$, and the almost-sure uniqueness of the conditional expectation, it follows that $\mu$ is compatible with $\pi$ if and only if 
$$
\mu(A \mid \cF_{\Lambda^c})(\omega) = \pi_\Lambda(A \mid \omega)
$$
holds for all $A \in \cF$ and $\Lambda \Subset \Z^2$, and for $\mu$-almost all $\omega$. Informally, this means that
  $$
  \mu(\tau \mid \tau_{\Lambda^c} = \omega_{\Lambda^c}) = \pi_\Lambda (\tau_\Lambda \mid \omega) \text{ for all }\Lambda \Subset \Z^2 \text{ and $\mu$-almost all } \omega.
  $$
Therefore, we will use $\mu_\Lambda(\cdot \mid \omega)$ to denote this conditional probability throughout this paper. 

By classic results \cite[Sec.~6.4, 6.7]{friedliLattice2018}, Gibbs measures exist under rather general assumptions which our setting certainly satisfies. Denote by $\mathscr{G}(\pi)$ the class of Gibbs measures compatible with the Gibbsian specification $\pi$. A condition involved in the definition of entropy is {\em translation invariance} of the
measure. Let $\theta_j:\Z^2\to\Z^2$ be a shift by $j\in \Z^2$, given by $\theta_j(i)=i+j$ for all $i$. A measure $\mu$ is called translation invariant if $\mu(A) = \mu(\theta_j^{-1} A)$ for all $j \in \Z^2$ and $A \in \cF$. Not all Gibbs measures that
arise from a translation-invariant potential retain this property, but some do; in particular, at least one of the measures
in $\mathscr{G}(\pi)$ is translation invariant. Translation-invariant measures are convenient because they support the use
of ergodic properties in defining the entropy rate.

The final concept that we need to mention is the {ground state} of the potential. For a pair $\omega,\tilde\omega \in \Omega$, let us define the \emph{relative Hamiltonian} by 
$$
    \mathscr{H}_\Phi(\tilde{\omega} \mid \omega) := \sum_{B \Subset \Z^2} \{\Phi_B(\tilde{\omega}) - \Phi_B(\omega)\}.
$$
A configuration $\eta \in \Omega$ is called a {\em ground state} for the potential $\Phi$ 
if $\mathscr{H}_\Phi(\omega | \eta) \ge 0$ for all $\omega \inftyeq \eta$ ($\omega$'s different from $\eta$ in at most finitely many sites).
In other words, ground states are configurations such that modifying a finite number of sites cannot decrease the Hamiltonian. 
By~\eqref{eq:specification}, the system favors configurations with low energy.
In particular, at zero temperature ($\beta = \infty$), the collection of ground states (configurations of the recoverable system) has probability 1. 

    To define Gibbs measures that are related to the recoverable system $\cX_0$, we translate the recovery rule~\eqref{eq:recovery} to potentials as follows.
    We say that $B \Subset \Z^2$ is a cross shape if $B$ is formed by five adjacent sites, $(i_1,i_2)$, $(i_1 \pm 1, i_2)$ and $(i_1, i_2 \pm 1)$. For a cross in the grid with center at $i$, we use the notation $\cross_i$ and write $\cross$ if the center does not matter.
        Define
    \begin{equation}\label{eq:potential}
        \Phi_B(\omega) = \begin{cases}
        - \beta & \text{if $B=\cross$ and $\omega_B$ matches the rule \eqref{eq:recovery},}\\
        \beta & \text{if $B=\cross$ and $\omega_B$ does not match the rule \eqref{eq:recovery},}\\
        0 & \text{otherwise}
    \end{cases}
    \end{equation}
where $\beta>0$ corresponds to the inverse temperature. 

\medskip
 Later in this paper, we separately consider the hard-core model on MISs, the definitions for which are provided in Section~\ref{sec:external_field}.

\section{High-temperature disordered regime}\label{sec:high-temp}

\subsection{Uniqueness of the Gibbs measure}

Generally, the Gibbs measure is not unique: for instance, a phase transition occurs in the 2D Ising model with no external magnetic field. At the same time, it is often unique for high temperature (small $\beta$). In this section, we estimate $\beta$ for which there is a unique Gibbs measure for the Hamiltonian $\mathscr{H}_{\Lambda, \Phi}$ defined above.

Define the \emph{influence} $I_{j \to i}$ by 
\begin{equation*}
    I_{j \to i} := \sup_{\substack{\omega, \omega' \in \Omega \\ \omega_k = \omega'_k \forall k \ne j}} \Vert \pi_i(\cdot \mid \omega) -\pi_i(\cdot \mid \omega') \Vert_{TV},
\end{equation*}
where 
$$\Vert \pi_i(\cdot \mid \omega) -\pi_i(\cdot \mid \omega') \Vert_{TV} = \frac{1}{2} \left (|\pi_i(0 \mid \omega) -\pi_i(0 \mid \omega')| + |\pi_i(1 \mid \omega) -\pi_i(1 \mid \omega')| \right )$$
is the total variation distance, and $\pi_i=\pi_{\{i\}}$ is a one-vertex distribution, \eqref{eq:specification}.
{\em Dobrushin's uniqueness condition} (see, e.g.,~\cite[Theorem~6.31]{friedliLattice2018}) states that $|\mathscr{G}(\pi)| = 1$ once
\begin{equation}\label{eq:uniqueness}
    \alpha := \sup_{i \in \Z^2} \sum_{j \in \Z^2} I_{j \to i} < 1.
\end{equation}
\begin{proposition}
    For the potential $\Phi = \{\Phi_B\}_{B \Subset \Z^2}$ defined by~\eqref{eq:potential},  inequality~\eqref{eq:uniqueness} holds for all $\beta < \beta_0$, where $0.049 < \beta_0 < 0.05$. 
    In particular, $\alpha \approx 0.68$ given $\beta = 0.049$.
\end{proposition}

\begin{proof}
The one-vertex probability $\pi_i$ is computed as follows 
\begin{align*}
     \pi_i (\eta_i \mid \omega)  :\!&= \frac{\exp(-\sH_{i;\Phi}(\eta_i \omega_{i^c}))}{\exp(-\sH_{i;\Phi}(0 \omega_{i^c}) ) + \exp(-\sH_{i;\Phi}(1\omega_{i^c} ))}      \\[.1in]
    & = \frac{\exp(-\sum_{B \Subset \Z^2} \Phi_B(\eta_i \omega_{i^c}))}{\exp(-\sum_{B \Subset \Z^2} \Phi_B(0 \omega_{i^c})) + \exp(-\sum_{B \Subset \Z^2} \Phi_B(1 \omega_{i^c}))} \\[.1in]
    & = \frac{e^{-\sum_{\substack{B: i\in B}} \Phi_B(\eta_i \omega_{i^c}) -\sum_{\substack{B: i\not\in B}} \Phi_B(\eta_i \omega_{i^c})}}{e^{-\sum_{\substack{B: i\in B}} \Phi_B(0 \omega_{i^c}) -\sum_{\substack{B: i\not\in B}} \Phi_B(0 \omega_{i^c})} + e^{-\sum_{\substack{B: i\in B}} \Phi_B(1 \omega_{i^c}) -\sum_{\substack{B: i\not\in B}} \Phi_B(1 \omega_{i^c})}} \\[.1in]
    & = \frac{\exp \left(-\sum_{\substack{B: i\in B}} \Phi_B(\eta_i \omega_{i^c})\right)}{\exp \left (-\sum_{\substack{B: i\in B}} \Phi_B(0 \omega_{i^c}) \right ) + \exp \left( -\sum_{\substack{B: i\in B}} \Phi_B(1 \omega_{i^c}) \right)} \qquad(\text{where }B\Subset \Z^2).
 \end{align*}
Therefore, $\pi_i(\cdot \mid \omega)$ depends only on those $\omega_j$ for which the pair of sites $i=(i_1,i_2), j=(j_1,j_2)$  are contained in some cross-shape $\cross$.
There are $12$ such points $j$, namely those given by
   $\{j \in \Z^2 \mid 0 < {\Vert i - j \Vert}_1 \le 2\}.
   $
The influence $I_{j \to i}$ can then be computed by exhausting all possible $\omega, \omega' \in \Omega$.
Moreover, $I_{j \to i} = 0$ whenever ${\Vert i - j \Vert}_1>2$. 
Since the system is shift invariant, we conclude that
   $$
   \alpha = \sup_{i \in \Z^2} \sum_{j \in \Z^2} I_{j \to i} = \sum_{0 < {\Vert j \Vert}_1 \le 2} I_{j \to 0}
   $$This quantity can be approximated by computer, and we obtain the claimed result.     
\end{proof}

\begin{remark}

An implication of the uniqueness of the Gibbs measure is the absence of long-range order. Let $U, V \Subset \Z^2$. The model specified by $\mu$ is considered disordered if, when $U$ is sufficiently distant from $V$, the value of $\omega_V$ has minimal influence on the marginal distribution on $U$, denoted by $\mu_U(\cdot \mid \omega_V)$.
To quantify this idea, we say that $\mu$ satisfies strong spatial mixing with (positive) constants $C,c$ if for any $U,V \Subset \Z^2$ and $\omega_V, \bar{\omega}_V \in \Omega_V$, we have
\begin{equation*}
   \tvd{\mu_U(\cdot \mid \omega_V) - \mu_U(\cdot \mid \bar{\omega}_V)} \le C |U| e^{-c d(U,V')},
\end{equation*}
where $d(\cdot,\cdot)$ is the $\ell_1$-distance in $\Z^2$ and $V' = \{v \in V, \omega_v \ne \bar{\omega}_v\}$ is the disagreement set.

 Since spatial mixing is defined and studied for finite graphs~\cite{peled2020three}, we can think that $U, V$ are contained in some large box $\Lambda \Subset \Z^2$, and let $\Lambda$ converge to $\Z^2$ as described next. 
 A sequence of finite regions $(\Lambda_n)_n$ is said to converge to $\Z^2$ in the sense of van Hove if $\lim_{n \to \infty} |\partial^{\text{in}} \Lambda_n|/|\Lambda_n| = 0$, where $\partial^{\text {in}} \Lambda := \{i \in \Lambda, \exists j \notin \Lambda, i \sim j\}$. Denote this convergence by $\Lambda_n\Uparrow\Z^2$ (dependence on $n$ is often suppressed).

We say that a measure $\mu$ is {\em fully supported} if $\mu_\Lambda(\omega_\Lambda) > 0$ for all $\omega_\Lambda \in \{0,1\}^\Lambda$ and $\Lambda \Subset \Z^2$. Further, $\mu$ has a {\em finite interaction range} if
$I_{i \to j} = 0$ whenever $d(i,j)>R$ for some finite $R$.
The next theorem shows that Dobrushin's uniqueness condition implies strong spatial mixing.
\begin{theorem}~{\rm\cite[Theorem~1.1]{peled2020three}}\label{thm: Spatial mixing}
    Let $\mu$ be a fully supported probability measure with finite interaction range $R$. 
Suppose that~\eqref{eq:uniqueness} holds. Then $\mu$ satisfies strong spatial mixing with $C = 1$ and
    $c = - \frac{1}{R}\log \alpha$; in other words, 
    \begin{equation*}
        \tvd{\mu_U(\cdot \mid \omega_V) - \mu_U(\cdot \mid \bar{\omega}_V)} \le |U| \alpha^{d(U, V')/R}.
    \end{equation*}
\end{theorem}
To draw conclusions from this result for our recoverable system $\cX_0,$ recall that for the Gibbs measure defined by 
\eqref{eq:specification}, \eqref{eq:potential}, the interaction radius is $R = 2$ by the Markov condition, a fact that follows by the Hammersley-Clifford theorem (e.g., \cite[Theorem~7.12]{grimmett2018probability}). 
Thus, Theorem~\ref{thm: Spatial mixing} applies, showing that our recoverable system does not have long-range interactions.
 \hfill$\triangleleft$
\end{remark}

\subsection{Information storage and error probability}\label{sec: error}

\subsubsection{Mixing rate}
Writing the data relying on recoverable configurations supports restoring the value of the site (or creating a replacement site) when it is lost due to physical degradation. 
At the same time, the recovery rule involves interdependence of the sites, and if the restoration cycle is performed many times, the stored information may be irrevocably lost. 
The number of steps before this happens depends on the temperature of the system: at high temperature it becomes more disordered, and information retention is less likely. 
In this section, we quantify this remark by modeling the recovery cycles as steps of a Markov process, namely a discrete-time Glauber chain. 

Fix a finite region $\Lambda\Subset \Z^2$. Let $\omega^e$ ($\omega^o$) be the configuration such that $\omega_i=1$ if 
$i_1+i_2$ is even (resp., odd) and $0$ otherwise.  
Clearly both $\omega^e$ and $\omega^o$ belong to $\cX_0$. Let $\Omega_\Lambda^e=\{\omega \mid \omega_{\Lambda^c}=\omega^e_{\Lambda^c}\}$
be the set of configurations that match the even boundary condition, and let $\Omega_\Lambda^o$ be the same for the odd boundary condition.

Let $\omega\in \Omega_\Lambda^e$. Below, we view the restricted
configuration $\omega_{\Lambda}$ as the place to store information and consider the evolution of the recovery process applied to it. The process evolves in discrete time, updating the configuration at time $t-1$, $X_{t-1}$, to $X_t$, where $X_0 := \omega$ is the initial state. At time $t$, the process chooses a site $j \in \Lambda$ uniformly at random and updates the value stored at $j$ by sampling from the one-vertex probability distribution $\pi_j(\cdot \mid X_{t-1})$, cf.~\eqref{eq:specification}.  Questions of this kind were previously considered in \cite{goldfeld2021Ising}.

The information capacity at time $t$ is defined by
\[
I_{\Lambda}^{(\beta)}(t) := \max_{p_{X_0}}I(X_0;X_t),
\]
where $p_{X_0}$ is the probability distribution of the initial configuration, and $I(\cdot;\cdot)$ is the mutual information.
This quantity characterizes the number of bits that can be stored in the system after $t$ corrections if we assume that errors are uniformly random, independently located, and no two errors happen simultaneously.
\begin{definition}[Mixing time]
    Let $\{X_t^{x_0}\}_{t=0}^{\infty}$ be a finite state Markov chain starting with state $x_0 \in \cX_{{\Lambda}}$ with transition matrix $P$
    and stationary distribution $\pi$, where $\cX_\Lambda$ is the state space, $|\cX_{{\Lambda}}| < \infty$.
    Let $P^t(x_0,\cdot)$ be the probability distribution of the chain at time $t$, and denote by $X_t^{x_0}$ the random state at time $t$. 
    Define
    \begin{equation*}
        d(t) := \max_{x \in \cX_\Lambda} \| P^t(x, \cdot) - {\pi} \|_{TV}.
    \end{equation*}
    The mixing time is defined by 
    \begin{equation*}
        t_{\rm mix}(\epsilon) = \min \{t : d(t) \le \epsilon\}.
    \end{equation*}
\end{definition}

Let $\{X_t^{x}\}_{t=0}^{\infty}$, $\{X_t^{y}\}_{t=0}^{\infty}$ be two Glauber chains of our system, where $x,y$ are the initial configurations.
If $t \ge t_{\text{mix}}(\epsilon)$, then $\| P^t(x,\cdot) - P^t(y, \cdot) \|_{TV} \le 2 \epsilon$ by the triangle inequality,
and therefore,  $|P^t(x,z)-P^t(y,z)|\le 4\epsilon$ for all $z\in \cX_\Lambda$.
Relying on this, one can show that $H(X_t|X_0)\approx H(X_t)$ and $I(X_0;X_t)\approx 0$, irrespective of $p_{X_0}$.

In conclusion, after $t_{\text{mix}}(\epsilon)$ steps of the Glauber chain, the configuration is almost independent of the initial state, i.e., $I_\Lambda^{(\beta)}(t_{\text{mix}})$ is small, and the system does not retain even a single bit of information. This establishes a negative result: if the storage device contains $n$ sites and the system operates at high
temperature, information is not retained beyond $t_{\text{mix}}(\epsilon)$ recovery cycles. Formally,
this argument is summarized in the following theorem.

\begin{theorem}\label{thm:fast-mixing}
    Consider the Glauber dynamics for the finite recoverable system on $\Lambda \Subset \Z^2$ with ${|\Lambda|=}n$ vertices and let $c_\beta = 13 - \frac{24}{1+e^{-{10}\beta}} > 0$ (thus $\beta < \frac{1}{{10}} \ln \frac{13}{11} \approx {0.0167}$), then 
    \begin{equation}\label{eq:mixing_time}
        t_{\mathrm{mix}}(\epsilon) \le \left \lceil \frac{n}{c_\beta} \left(\ln n + \ln \frac{1}{\epsilon} \right) \right \rceil.
    \end{equation}
\end{theorem}

The proof relies on coupling of Markov chains and is given in Appendix~\ref{app:coupling}.

\subsubsection{Error probability} \label{sec:lotemp}

In this section, we estimate the error probability of recovery of a particular site in the configuration at low temperature. 
Let $\Lambda \Subset \Z^2$ be simply connected. 
For every configuration $\omega \in \Omega_{\Lambda }^e$, denote by $\cE(\omega) = \{ i \mid \Phi_{\cross_i}(\omega) = \beta\}$ the collection of sites in $\omega$ that do not satisfy the recovery condition. 
Note that $\cE(\omega) \subset \Lambda \cup \partial \Lambda$. 
To shorten the writing, let us denote by $\Lambda^+ := \Lambda \cup \partial \Lambda$, and introduce the notation $\varepsilon_\omega(A) = |\cE(\omega) \cap A|$ for $A \Subset \Z^2$.
The Hamiltonian $\sH_{\Lambda; \Phi}(\omega)=\sH_{\Lambda ; \beta}(\omega)$ takes the form
\begin{equation}\label{eq:hami}
    \sH_{\Lambda; \beta} (\omega) = - \beta |\Lambda^+| + 2\beta \varepsilon_\omega(\Lambda^+).
\end{equation}
The Gibbs distribution with even boundary condition is given by 
\begin{equation}\label{eq:Gibbs}
    \mu_{\Lambda;\beta}^e(\omega) := \frac{e^{-\sH_{\Lambda;\beta}(\omega)}}{Z_{\Lambda;\beta}^e},
\end{equation}
where the partition function is 
\begin{equation}\label{eq:partition}
\begin{aligned}
    Z_{\Lambda ;\beta}^e &= \sum_{\omega \in \Omega_\Lambda^e} e^{-\sH_{\Lambda;\beta}(\omega)}=
    e^{\beta|\Lambda^+|} \sum_{\omega\in\Omega_\Lambda^e} e^{-2\beta \varepsilon_\omega(\Lambda^+)}.
\end{aligned}
\end{equation}
Therefore, for $\omega \in \Omega_\Lambda^e$, we have 
\begin{equation}\label{eq:Gibbs1}
    \mu_{\Lambda;\beta}^e (\omega)= \frac{e^{-2\beta \varepsilon_\omega(\Lambda^+)}}
    {\sum_{\omega \in\Omega_\Lambda^e} e^{-2\beta\varepsilon_\omega(\Lambda^+)}}
\end{equation}
In the following proposition, we limit ourselves to errors at sites within $\Lambda$.
\begin{proposition}\label{prop:local_pertarbation} Let $\Lambda \Subset \Z^2$ be simply connected.
For all $\beta > 0$ and $i \in \Lambda$, there exists a constant $C \le 16$ such that 
$$
\mu_{\Lambda;\beta}^e(\{\omega \in \Omega_{\Lambda}^e \mid i \in \mathcal{E}(\omega)\}) \le C e^{-2\beta}.
$$
\end{proposition}
\begin{proof}
For $i\in \Lambda$ we have
\begin{equation}\label{eq:Pe}
\mu_{\Lambda;\beta}^e(\{\omega \in \Omega_{\Lambda}^e \mid i \in \mathcal{E}(\omega)\})
=e^{-2\beta}\frac{\sum_{\omega \in \Omega_{\Lambda}^e : i \in \cE(\omega)}e^{-2\beta \varepsilon_\omega(\Lambda^+ \backslash\{i\})}}
{\sum_{\omega \in\Omega_\Lambda^e} e^{-2\beta \varepsilon_\omega(\Lambda^+)}},
\end{equation}
where the leading term $e^{-2\beta}$ is due to the fact that $i \in \cE(\omega) \cap \Lambda$. 

We will show that the fraction on the RHS in \eqref{eq:Pe} is bounded above by a constant $C \le 16$. 
To this end, we construct, for each site $i \in \Lambda$, a function $f_i:\Omega_{\Lambda}^{e} \to \Omega_{\Lambda}^{e}$ that, given a configuration $\eta \in \Omega_{\Lambda}^e$ with $i \in \cE(\eta) \cap \Lambda$, modifies $\eta$ only within the cross $\cross_i$, yielding a new configuration $\omega = f_i(\eta)$ with $\cE(\omega) \subseteq \cE(\eta) \setminus \{i\}$. Then, taking $C = \max\limits_{i \in \Lambda, \omega \in \Omega_\Lambda^e} |f_i^{-1}(\omega)|$, we would get the inequality
$$
\sum_{\eta \in \Omega_\Lambda^e : i \in \cE(\eta)} e^{-2\beta \varepsilon_\eta(\Lambda^+ \backslash\{i\})}
\le C\sum_{\omega\in \Omega_\Lambda^e}e^{-2 \beta \varepsilon_\omega(\Lambda^+)}.
$$
It remains to construct $f_i$. 
Given $\eta \in \Omega_{\Lambda}^e$ with $i \in \cE(\eta)\cap\Lambda$, we define $\omega = f_i(\eta)$ through the following steps:
\begin{itemize}
\item[(1)] Set $\omega \leftarrow \eta$.
\item[(2)] If $\omega_i = 0$, set $\omega_i \leftarrow 1$ and go to Step (4).
\item[(3)] If $\omega_i = 1$:
\begin{itemize}
\item[(3a)] set $\omega_i \leftarrow 0$;
\item[(3b)] for all $j \in \cross_i \cap \cE(\omega)$ such that $\omega_j = 0$, set $\omega_j \leftarrow 1$.
\end{itemize}
\item[(4)] Return $f_i(\eta) = \omega$.
\end{itemize}
We claim that $\cE(\omega) \subseteq \cE(\eta) \setminus \{i\}$. To see this, first observe that if an error at any site $j$ is due to $\omega_j$ being $0$, then flipping $\omega_j$ corrects that error without creating any new errors. If the condition in Step (2) is met, then we have $\cE(\omega) \subseteq \cE(\eta) \setminus \{i\}$. If instead, we are in the situation of Step~(3), then the error at site $i$ is corrected by flipping $\omega_i$, but in the process we may create new errors in the neighboring sites $j \in \cross_i \setminus \{i\}$. Note, however, that no \emph{new} error is created at a neighboring site $j$ with $\omega_j = 1$. (Indeed, if such a site $j$ is in $\cE(\omega)$ after setting $\omega_i \leftarrow 0$, it must have been in error even before $\omega_i$ was flipped, i.e., it must have been the case that $j \in \cE(\eta)$.) 
On the other hand, the act of flipping $\omega_i$ may create new errors at neighboring sites $j$ such that $\omega_j = 0$. But each such error can be corrected, without creating any further errors, by flipping $\omega_j$ in Step~(3b). It follows that $\cE(\omega) \subseteq \cE(\eta) \setminus\{i\}$, as claimed.
To see that $f_i(\omega) \in \Omega_\Lambda^e$, note that the described procedure does not make any changes
outside $\Lambda$. Indeed, it does not examine any sites in $(\Lambda\cup\partial\Lambda)^c$ and it does not
change the values $\eta_j,j \in \partial\Lambda$ since every $0$ in $\Lambda^c$ is adjacent to at least one $1$ in $\Lambda^c$. Thus $j \notin \cE(\omega)$. 

Finally, we claim that $|f_i^{-1}(\omega)| \le 16$ for any $\omega \in \Omega_{\Lambda}^e$. This is because any $\eta \in f_i^{-1}(\omega)$ differs from $\omega$ only in the five sites within the cross $\cross_i$, and moreover, all $\eta \in  f_i^{-1}(\omega)$ take the same value at the center of the cross: $\eta_i = 1 \oplus \omega_i$, where $\oplus$ is mod-2 addition.
From this, we conclude that $C = \max\limits_{i,\omega} |f_i^{-1}(\omega)| \le 2^4 = 16$.
\end{proof}

As a conclusion, at sufficiently low temperatures, the expected size of the error set $\cE(\cdot)$ is small, and so typical configurations mostly arise as local perturbations of exactly recoverable configurations $\omega\in \cX_0$.

\section{Entropy of the system}\label{sec:entropy}
For a finite system on $\Lambda \Subset \Z^2$ and a boundary condition $\omega$, the entropy rate is defined by 
\begin{equation}\label{eq:top1}
    H\left (\pi_{\Lambda}(\cdot \mid \omega) \right ) := - \frac{1}{|\Lambda|} \sum_{\tau_\Lambda \in \Omega_\Lambda} \pi_{\Lambda}(\tau_\Lambda \mid \omega) \log_2 \pi_{\Lambda}(\tau_\Lambda \mid \omega).
\end{equation}

\subsection{Zero temperature}
At zero temperature, a Gibbs measure is supported on the set of ground states. 
% \Gnote{By eq~\eqref{eq:potential}, ground states have $-\infty$ Hamiltonian.}
Let $\cG_{\Lambda}$ be the set of the ground state configurations restricted to $\Lambda$.
The zero-temperature entropy (also called the topological entropy) is defined as the limit
\begin{equation}\label{eq:topo_entropy}
    H_0:= \lim_{m,n \to \infty}\frac{\log_2 |\cG_{\Z_{m \times n}}|}{mn},
 \end{equation}
where $\Z_{m \times n}$ is the $m\times n$ box $\{-\lfloor\frac m2\rfloor,\dots,\lceil\frac m2\rceil - 1\}\times\{-\lfloor\frac n2\rfloor,
\dots,\lceil\frac n2\rceil - 1\}$. 
If $\mu$ is a translation-invariant {\em zero-temperature Gibbs measure}, then $H(\mu) \le H_0$. 
Moreover, there exists a translation-invariant zero-temperature Gibbs measure $\mu$ such that $H(\mu) = H_0$ 
(see, e.g., \cite[Appendix B2]{van1993regularity} for a detailed discussion of the definitions).

We begin with a general characterization of the entropy $H_0$.
Let us fix $\omega \in \cX_0$ and denote by $\cG^\omega_\Lambda$ the set of configurations $\tau$ that have no recovery rule violations within $\Lambda\cup\partial\Lambda$ under the constraint $\tau_{\Lambda^c} = \omega_{\Lambda^c}$.
The conditional restriction of a zero-temperature Gibbs measure on $\Lambda$ is the uniform distribution on $\cG^\omega_{\Lambda}$. With this, the entropy rate expression in \eqref{eq:top1} simplifies at zero temperature to
\begin{equation}\label{eq:zero_temp_entropy}
    H\left (\pi_{\Lambda}(\cdot \mid \omega) \right ) := \frac{1}{|\Lambda|} \log_2 |\cG^\omega_{\Lambda}|.
\end{equation}

Setting $\Lambda = \Z_{m \times n}$, the expression on the right in \eqref{eq:zero_temp_entropy} becomes $\frac{1}{mn} \log_2 |\cG^\omega_{\Z_{m \times n}}|$. The proposition below relates this to $H_0$ in \eqref{eq:topo_entropy}. In what follows, we write $\cG^\omega_{m\times n}$ (resp.\ $\cG_{m\times n}$) instead of $\cG^\omega_{\Z_{m \times n}}$ (resp.\ $\cG_{\Z_{m \times n}}$) to save on notation.

\begin{proposition}\label{prop:uniqueness_limit}
        For every boundary condition $\omega$ $\in \cX_0$, we have
        \begin{equation*}
            \lim_{m,n\to \infty} \frac{\log_2 |\cG^\omega_{{m \times n}}|}{mn} = H_0.
        \end{equation*}
\end{proposition}
\begin{proof}
    Take $m,n$ to be suitably large.
     Let $k>2$ be a constant and denote by $\cG^\omega_{(m \times n);k}$ the  restriction of $\cG^\omega_{m \times n}$ to the box $\Z_{(m-k) \times (n-k)}$. We then have $\cG^\omega_{(m \times n);k} \subseteq \cG_{(m-k) \times (n-k)}$.
    Therefore,
    \begin{equation*}
        \begin{split}
            |\cG^\omega_{m \times n}| & \le 2^{O(k(m+n))} |\cG^\omega_{(m \times n);k}| \\
            & \le 2^{O(m+n)} |\cG_{(m-k) \times (n-k)}| \\
            & \le 2^{O(m+n)} |\cG_{m \times n}|.
        \end{split}
    \end{equation*}
    Taking logarithms on both sides and sending $m,n \to \infty$, we obtain 
   \begin{align*}
        \lim_{m,n\to \infty} \frac{\log_2 |\cG^\omega_{{m \times n}}|}{mn} 
        &\le \lim_{m,n \to \infty} \frac{\log_2|\cG_{m \times n}|}{mn} = H_0.
    \end{align*}
    On the other hand, 
    \begin{equation*}
        \lim_{m,n \to \infty} \frac{\log_2 |\cG^\omega_{{m \times n}}|}{\log_2 |\cG_{{m\times n}}|} 
            \ge \lim_{m,n \to \infty} \frac{\log_2 |\cG_{(m-k)\times (n-k)}|}
            {\log_2 |\cG_{m\times n}|} = 1. \qedhere
    \end{equation*}

%\end{multline*}
\end{proof}
This proposition shows that for regions $\Lambda$ of the form $\Z_{m \times n}$, the entropy rate in~\eqref{eq:zero_temp_entropy} approaches $H_0$ as $\Lambda \Uparrow \Z^2$, independently of the boundary condition $\omega$.

\begin{theorem}\label{thm:H_0 bound}
    The topological entropy at zero temperature satisfies $0.3012 \le H_0 \le 0.3409$.
\end{theorem}
The bound is obtained by counting the number of MISs on $\Z_{m \times n}$, which is based on a standard application of the transfer matrix method, given in Appendix~\ref{App:zero-temp-entropy}.

\subsection{Positive temperature} \label{sec: H positive}
At positive temperature (finite $\beta$), the entropy of the system should increase from $H_0$. 
In this section, we show that this is indeed the case by giving a lower bound on $ H(\mu_{U_n; \beta}^e)$ 
that depends on $\beta$, where $U_n$ is the box $\Z_{(2{n+1}) \times (2{n+1})} = \{-n,\ldots,n\}^2$.
From \eqref{eq:top1} we have 
\begin{align}
        H(\mu_{U_n; \beta}^e) &:=-\frac 1{|U_n|}\sum_{\omega\in\Omega_{U_n}^e}
        \mu_{U_n; \beta}^{e}(\omega) \log_2\mu_{U_n; \beta}^{e}(\omega) \nonumber\\
        &= - \frac{1}{|U_n|} \sum_{\omega \in \Omega_{U_n}^e} \mu_{U_n; \beta}^{e}(\omega) \log_2 \frac{e^{-\sH_{U_n; \beta}(\omega)}}{\bfZ_{U_n; \beta}^e} \nonumber\\
        & = \frac{\langle \sH_{U_n; \beta} \rangle_{\mu_{U_n; \beta}^e}}{|U_n|} \log_2 e + \frac{\log_2 \bfZ_{U_n; \beta}^e}{|U_n|},\label{eq:entropy_mu} 
        \end{align}
where $\sH_{U_n; \beta}$ is given in~\eqref{eq:hami} and $\langle \cdot \rangle_{\mu_{U_n; \beta}^e}$ is the expectation with respect to the measure $\mu_{U_n; \beta}^e$ defined in \eqref{eq:Gibbs}. Recalling our notation, $U_n^+$ refers to $U_n$ and its outer boundary.

To compute the first term in \eqref{eq:entropy_mu}, we find
\begin{align}\label{eq:Ehami}
    \langle \sH_{U_n; \beta} \rangle_{\mu_{U_n; \beta}^e}
    & =  - \beta |U_n^+| + 2 \beta \sum_{i \in U_n^+}
          \mu_{U_n; \beta}^e \left( \{\omega \in \Omega_{U_n}^e \mid i \in \cE(\omega)\} \right) \nonumber\\
    & \ge - \beta |U_n^+| + 2 \beta \sum_{i \in U_n}
          \mu_{U_n; \beta}^e \left( \{\omega \in \Omega_{U_n}^e \mid i \in \cE(\omega)\} \right) .
\end{align}
where the inequality follows by restricting the sum to $U_n \subset U_n^+$.
\begin{lemma}\label{lemma:error-low}
    For all $\beta>0$ and $i \in U_n$, 
    \begin{equation}\label{eq:recovery_prob_low}
        \mu_{U_n; \beta}^e \left( \{\omega \in \Omega_{U_n}^e \mid i \in \cE(\omega) \} \right) \ge \frac{1}{1 + e^{10 \beta}}.
    \end{equation}
\end{lemma}

\begin{proof} Throughout the proof, we implicitly assume that $i \in U_n$.
Rewriting ~\eqref{eq:Gibbs}, 

\begin{equation*}
    \begin{split}
    \mu_{U_n; \beta}^e \left( \{\omega \in \Omega_{U_n}^e \mid i \in \cE(\omega) \} \right) 
     & = \frac{\sum_{\omega \in \Omega_{U_n}^e \!\!: i \in \cE(\omega)} e^{-2 \beta \varepsilon_\omega(U_n^+)}}{\sum_{\omega \in \Omega_{U_n}^e} e^{-2 \beta \varepsilon_\omega(U_n^+) }} \\
     & = \frac{\sum_{\omega \in \Omega_{U_n}^e \!\!: i \in \cE(\omega)} e^{-2 \beta \varepsilon_\omega(U_n^+)}}{\sum_{\omega \in \Omega_{U_n}^e \!\!: i \in \cE(\omega)} e^{-2 \beta \varepsilon_\omega(U_n^+)} + \sum_{\omega \in \Omega_{U_n}^e \!\!: i \notin \cE(\omega)} e^{-2 \beta \varepsilon_\omega(U_n^+)}}.
    \end{split}
\end{equation*}

Let $f_i:\Omega_{U_n}^e\to\Omega_{U_n}^e$ be the mapping that flips the bit stored at site $i$. 
This mapping gives a bijection between the configurations in which the $i$-th bit does not match the recovery rule and those in which it does. Moreover, for $\omega \in \Omega_{U_n}^e$ with $i \in \cE(\omega) \cap U_n$ we have $\varepsilon_{f_i(\omega)}(U_n^+) \ge \varepsilon_\omega(U_n^+) -5$.
In other words, the flip operation reduces the number of errors by at most $5$.
To see this, suppose that $\omega_j = 0$ for all $j$ with $\|i-j\|_1 \le 2$ and that there are $5$ errors around $i$. Since $f_i(\omega)$ flips the middle $0$ to $1$, this corrects $5$ errors.
Therefore,

\begin{align*}
       \mu_{U_n; \beta}^e \left( \{\omega \in \Omega_{U_n}^e \mid i \in \cE(\omega)\} \right) 
       &\ge \frac{\sum_{\omega \in \Omega_{U_n}^e \!\!: i \in \cE(\omega)} e^{-2 \beta \varepsilon_\omega(U_n^+)}}{\sum_{\omega \in \Omega_{U_n}^e \!\!: i \in \cE(\omega)} e^{-2 \beta \varepsilon_\omega(U_n^+)} + e^{10\beta}\sum_{\omega \in \Omega_{U_n}^e \!\!: i \in \cE(\omega)} e^{-2 \beta \varepsilon_\omega(U_n^+)}}  \\
    & = \frac{1}{1+e^{10\beta}}.
\end{align*}
\end{proof}
Combining \eqref{eq:Ehami} and~\eqref{eq:recovery_prob_low}, we obtain
\begin{equation}\label{eq:energy_density}
    \frac{\langle \sH_{U_n; \beta} \rangle_{\mu_{U_n; \beta}^e}}{|U_n|} \ge -\beta \frac{|U_n^+|}{|U_n|} + \frac{2\beta}{1+e^{10\beta}} \gtrsim -\beta + \frac{2\beta}{1 + e^{10 \beta}}\quad \text{ as } n \to \infty.
\end{equation}
From~\eqref{eq:partition} we have
\begin{equation} \label{eq:log_Z}
\begin{split}
    \log_2 \bfZ_{U_n; \beta}^e & = \beta |U_n^+| \log_2 e + \log_2 \sum_{k = 0}^{|U_n^+|} a_{U_n}(k) e^{-2k\beta},
\end{split}
\end{equation}
where $a_{U_n}(k)=|\{\omega\in\Omega_{U_n}^e : \varepsilon_\omega(U_n^+) =k \}|$ is the number of configurations on $U_n$ with exactly $k$ unmatched bits (under the even boundary condition). 
Note that $\log_2 \sum_{k=0}^{|U_n^+|} a_{U_n}(k) e^{-2k\beta} \ge \log_2 a_{U_n}(0)$ and
\begin{equation}\label{eq:pressure1}
    \begin{split}
    \lim_{n\to\infty} \frac{\log_2 \bfZ_{U_n; \beta}^e}{|U_n|} & \ge \beta \log_2 e \lim_{n \to \infty} \frac{|U_n^+|}{|U_n|} + \lim_{n\to\infty}\frac{\log_2 a_{U_n}(0)}{|U_n|} = \beta \log_2 e + H_0.
    \end{split}
\end{equation}

On the other hand, by~\eqref{eq:log_Z}
\[
\begin{split}
 \log_2 \sum_{k=0}^{|U_n^+|} a_{U_n}(k) e^{-2k\beta}&\ge \log_2 \left [ a_{U_n}(0) + (2^{|U_n|} - a_{U_n}(0)) e^{-2 \beta |U_n^+|}  \right ] \\
&  = \log_2 a_{U_n}(0) + \log_2 \left [ 1 + \left( \frac{2^{|U_n|}}{a_{U_n}(0)} -1 \right ) e^{-2\beta |U_n^+|} \right ] \\
&  \ge \log_2 a_{U_n}(0) + \log_2 \left [ \frac{2^{|U_n|}}{a_{U_n}(0)} e^{-2\beta |U_n^+|}\right]  \\
& = |U_n| - 2 \beta |U_n^+| \log_2 e,
\end{split}
\]
where the first inequality comes from the observation that
\begin{equation*}
    |\Omega_{U_n}^e| = 2^{|U_n|} = \Big | \bigcup_{k=0}^{|U_n^+|} \{ \omega \in \Omega_{U_n}^e \mid \varepsilon_\omega(U_n^+) = k \} \Big | = \sum_{k=0}^{|U_n^+|} a_{U_n}(k).
\end{equation*}
Therefore,
\begin{equation}\label{eq:pressure2}
    \begin{split}
        \lim_{n\to\infty} \frac{\log_2 \bfZ_{U_n; \beta}^e}{|U_n|} \ge \beta \log_2 e \lim_{n \to \infty} \frac{|U_n^+|}{|U_n|}  + \Big (1 - 2 \beta \log_2 e \lim_{n \to \infty} \frac{|U_n^+|}{|U_n|} \Big ) =   1 - \beta \log_2 e.
    \end{split}
\end{equation}
\begin{theorem}[Entropy lower bound]\label{thm:entropy-low-bound}
For large $n$, the entropy of $\mu^e_{U_n,\beta}$ satisfies
    \begin{equation} \label{eq:entropy_bound}
        H(\mu^e_{U_n, \beta}) \gtrsim  \max \Big\{ \frac{2 \beta \log_2 e}{1+e^{10 \beta}} + H_0, 1 + \frac{2 \beta 
        \log_2 e}{1+e^{10 \beta}} - 2\beta \log_2 e \Big\}.
    \end{equation}
\end{theorem}
%\vspace*{.05in}
\begin{proof}
    This is straightforward by combining~\eqref{eq:entropy_mu}, \eqref{eq:energy_density},  \eqref{eq:pressure1}, and \eqref{eq:pressure2}.
\end{proof}
The second term in \eqref{eq:entropy_bound} is larger for small $\beta$, and the first term takes over as $\beta$ increases. In any case, 
  $$
  \limsup_{n\to\infty} H(\mu^e_{U_n,\beta})>H_0 \text{ for all }\beta<\infty.
  $$
 
\begin{figure}
    \centering
    \includegraphics[width=0.45\linewidth]{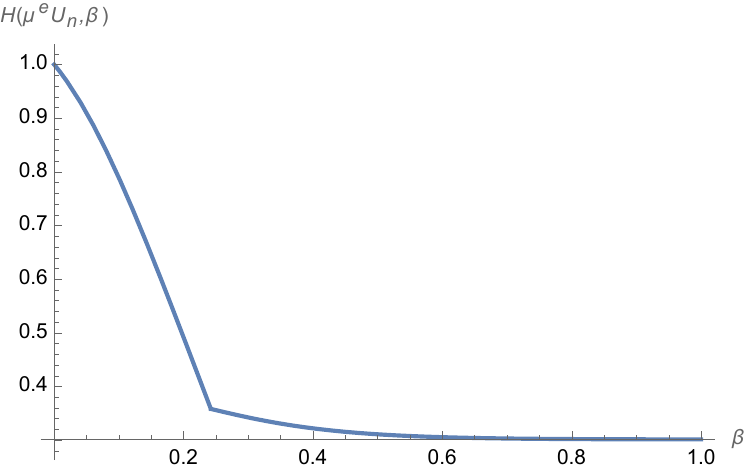}
     \caption{Plot of the bound \eqref{eq:entropy_bound}. 
     The horizontal axis corresponds to the entropy value $H_0 \approx 0.3012$ (the topological entropy), attained at zero temperature. 
     The singularity occurs at the $\beta$ for which the maximum switches from one term in~\eqref{eq:entropy_bound} to the other. }
     \label{fig:entropy}
\end{figure}

We point out the trade-off between data recovery probability (accuracy) and entropy (capacity). At zero temperature ($\beta = \infty$), the system is supported on $\cX_0$, the recovery is exact, and the entropy is $H_0$.
For large but finite $\beta$,  Proposition~\ref{prop:local_pertarbation} ensures that the recovery probability remains high, while Theorem~\ref{thm:entropy-low-bound} shows that the capacity is increased due to the inclusion of locally perturbed configurations.
At the opposite extreme, as $\beta \to 0$, the system becomes highly disordered. 
By Lemma~\ref{lemma:error-low}, the recovery probability approaches $\frac{1}{2}$, and by Theorem~\ref{thm:fast-mixing}, the system cannot reliably retain a single bit of information.

\section{A hard-core model for recoverable systems} \label{sec:external_field}

\subsection{Maximal independent sets}

\subsubsection{The hard-core model}

The well-known hard-core model was originally introduced in statistical physics as a simplified model of a gas. 
A configuration $\omega \in \Omega$ represents the occupancy of particles on the lattice $\Z^2$, where $\omega_i = 1$ indicates that site $i$ is occupied. The model imposes a hard-core constraint, forbidding adjacent occupied sites in admissible configurations. If $\omega$ is such a configuration, then $\omega_i \omega_j = 0$ for all $i \sim j$. 
Equivalently, the set of occupied sites in any valid configuration forms an independent set in $\Z^2$.
The collection of valid configurations has probability $1$ under the hard-core Gibbs measure.
Alternatively, we can define this measure on the set of independent sets $\sI(\Z^2)$.

\begin{definition}[Hard-core Gibbs measure on independent sets]
    We say that a probability measure $\mu$ is a \emph{hard-core (Gibbs) probability measure} with activity $\lambda > 0$ if for every $\Lambda \Subset \Z^2$, $J \in \sI(\Z^2)$ and $\mu$-almost every $I \in \sI(\Z^2)$, 
    \begin{equation}\label{eq:def_hard-core_ind}
        \mu(J \mid I \cap \Lambda^c) = \begin{cases}
            \frac{\lambda^{|J \cap \Lambda|}}{\bfZ_{\Lambda, \lambda}^I} & \text{if } I \cap \Lambda^c = J \cap \Lambda^c \\
            0 & \text{otherwise},
        \end{cases}
    \end{equation}
    where
    \begin{equation*}
        \bfZ_{\Lambda, \lambda}^I = \sum_{K \in \sI(\Z^2): K \cap \Lambda^c = I \cap \Lambda^c} \lambda^{|K \cap \Lambda|}
    \end{equation*}
    is the partition function.
\end{definition}

The parameter $\lambda > 0$, called the \emph{activity}, controls the particle density: larger $\lambda$ favors configurations with more occupied sites, subject to the hard-core constraint. For small values of $\lambda$, particles are sparse and disordered, and the Gibbs measure is known to be unique. As $\lambda$ increases, long-range correlations emerge and the system exhibits a phase transition.

It has long been conjectured~\cite{baxter1980hard} that the hard-core model on $\mathbb{Z}^2$ exhibits a critical activity $\lambda_c \approx 3.796$, such that for $\lambda < \lambda_c$ the Gibbs measure is unique, and for $\lambda > \lambda_c$, multiple extremal Gibbs measures coexist. The best known lower bound on $\lambda_c$ is due to Sinclair et al.~\cite{sinclair2017spatial}, who proved uniqueness for $\lambda < 2.538$. 
Earlier results on the lower bound for $\lambda_c$ include~\cite{radulescu1987dobrushin}, \cite{berg1994percolation}, \cite{weitz2006counting}, \cite{restrepo2013improved}. 
On the other hand, Blanca et al.~\cite{blanca2019phase} showed that the hard-core model admits a phase transition for $\lambda > 5.3506$, meaning that this is an upper bound on $\lambda_c$ (if the critical activity value exists).

\subsubsection{The maximal hard-core model} Inspired by~\eqref{eq:def_hard-core_ind}, we now define a natural analogue of the hard-core Gibbs measure supported on the space of MISs $\cX_0$.

\begin{definition}[The maximal hard-core model]
    A probability measure $\mu$ on  $\cX_0$ is called a \emph{maximal hard-core Gibbs measure} with activity $\lambda > 0$ if, for every finite set $\Lambda \Subset \Z^2$, $\eta \in \cX_0$ and $\mu$-almost every $\omega \in \cX_0$,
    \begin{equation}\label{eq:def_hard-core_mis}
        \mu_\Lambda(\eta \mid \omega) = \begin{cases}
            \frac{\lambda^{|\eta_{\Lambda}|}}{\bfZ_{\Lambda, \lambda}^\omega} & \text{if } \eta_{\Lambda^c} = \omega_{\Lambda^c}, \\
            0 & \text{otherwise},
        \end{cases}
    \end{equation}
    where $\Lambda^c = \Z^2 \setminus \Lambda$, $|\eta_\Lambda|:=|\supp(\eta_\Lambda)|$, and the partition function is given by
    \begin{equation*}
        \bfZ_{\Lambda, \lambda}^\omega = \sum_{\eta \in \cX_0: \eta_{\Lambda^c} = \omega_{\Lambda^c}} \lambda^{|\eta_\Lambda|}.
    \end{equation*}
\end{definition}

For completeness, we check in Appendix~\ref{App: existence of measure} that $\mu$ is well-defined, i.e., that maximal hard-core Gibbs measures always exist. 

\subsection{Phase coexistence at high activity}
In this section we establish the following theorem.

% \vspace*{.05in}
\begin{theorem}\label{thm:main}
    The maximal hard-core model exhibits a phase transition for all $\lambda > 6.3506$.
\end{theorem}
% \vspace*{.05in}
The general idea of the proof of this theorem follows by adapting to our case a Peierls-style argument developed by Blanca et al.~\cite{blanca2019phase} for the standard hard-core model. There are several important differences between our arguments and the proof in \cite{blanca2019phase}, which are summarized at the end of this section. 

We begin with an overview of the proof.  Recall our notation of the odd and even configurations, $\omega^o$ and $\omega^e$, from Section~\ref{sec: error}. 
 Recall also that $U_n = \{-n, \dots, n\}^2$ is a box.
Denote by $\cJ_n^e$ the set of all configurations contained in $\cX_0$ that agree with $\omega^e$ outside $U_n$, i.e.,
\[
    \cJ_n^e = \left\{ \eta \in \cX_0 \mid \eta_{U_n^c} = \omega^e_{U_n^c} \right\}.
\]
The corresponding Gibbs distribution on $U_n$ with even boundary condition is given by
\begin{equation*}
    \mu_n^e (\eta) := \mu_{U_n}(\eta \mid \omega^e) = \frac{\lambda^{|\eta_{U_n}|}}{\bfZ_n^e},
\end{equation*}
where $\eta\in \cJ_n^e$ and $\bfZ_n^e = \sum_{\eta \in \cJ_n^e} \lambda^{|\eta_{U_n}|}$.
The measure $ \mu_n^e $ is supported on $\cJ_n^e$. 
The Gibbs distribution $\mu_n^o$ under the odd boundary condition $\omega^o$ is defined analogously.

Let $\mu^e$, $\mu^o$ be the Gibbs measures obtained as weak subsequential limits of $\{\mu_n^e\}_n$ and $\{\mu_n^o\}_n$ respectively. We will show that the measures $\mu^e$ and $\mu^o$ are different. 

Throughout our argument, we will fix integers $m,n$ with $n > m$, and assume $m$ to be sufficiently large. 
Let $[m^-]:=\{-m-1, \dots, m\},[m]:=\{-m, \dots, m\},$ and $[m^+]:=\{-m, \dots, m+1\}$ and define the sets
    \begin{gather*}
    U_m^W=[m^-]\times [m],\quad U_m^E=[m^+]\times[m],\\
    U_m^S=[m]\times[m^-], \quad U_m^N=[m]\times[m^+]
    \end{gather*}
and refer to them collectively as $U_m^\ast$. Given a configuration $\eta \in \cX_0$, we say that 
\begin{itemize}
    \item $\eta$ is {$m$-odd} if the subset $\{i \in \Z^2 \mid \eta_i = \omega^o_i\}$ contains  at least one of the sets $U_m^\ast$,
    \item $\eta$ is {$m$-even} if the subset $\{i \in \Z^2 \mid \eta_i = \omega^e_i\}$ contains at least one of the sets $U_m^\ast$,
    \item $\eta$ is $m$-homogeneous if it is either $m$-even or $m$-odd.
\end{itemize}

To prove Theorem~\ref{thm:main}, it suffices to identify an event for which the conditional probabilities differ under the two limiting measures. Specifically, let $\sO_m$ denote the event that a configuration is $m$-odd, and $\sE_m$ the event that it is $m$-even. We will show that for all $n > m$ and sufficiently large $m$, 
\begin{equation}\label{eq:even_boundary}
    \mu_n^e (\sO_m \mid \sO_m \cup \sE_m)  \le \frac{1}{3},
\end{equation}
and 
\begin{equation} \label{eq:odd_boundary}
    \mu_n^o (\sO_m \mid \sO_m \cup \sE_m) \ge \frac{2}{3}
\end{equation}
(note that the choice of $1/3$ in~\eqref{eq:even_boundary} is arbitrary in the sense that the argument applies for any number in $(0,0.5)$).
Once these inequalities are established, we will have
\begin{equation} \label{eq:distinguish}
    \mu^e(\sO_m \mid \sO_m \cup \sE_m) < \mu^o(\sO_m \mid \sO_m \cup \sE_m),
\end{equation}
which implies that the limiting Gibbs measures depend on the boundary condition, confirming the presence of a phase transition. In other words, under the condition that a configuration is $m$-homogeneous, the even boundary condition $\omega^e$ implies that it is unlikely to be $m$-odd, and vice versa.

Let $\cB_{m,n}^e$ ($\cA_{m,n}^e$) be the set of $m$-odd (resp., $m$-homogeneous) configurations that are consistent with the even boundary condition $\omega^e$ outside $U_n$.
To prove~\eqref{eq:even_boundary}, we will construct a mapping $\phi : \cB_{m,n}^e \mapsto \cA_{m,n}^e \setminus \cB_{m,n}^e$ such that for a fixed large $m$ and every $\tau \in \cA_{m,n}^e \setminus \cB_{m,n}^e$,
\begin{equation}\label{eq:shift_bound}
    \sum_{\substack{\eta \in \cB_{m,n}^e \\ \phi(\eta) = \tau}} \lambda^{|\eta_{U_n}| - |\tau_{U_n}|} \le \frac{1}{3}
\end{equation}
If \eqref{eq:shift_bound} holds true, then for each $m$-odd configuration $\eta \in \cB_{m,n}^e$, there exists a corresponding $m$-even configuration $\phi(\eta)$ with significantly larger $\mu_n^e$ probability. Using 
\eqref{eq:shift_bound}, we find
\begin{equation*}
    \begin{split}
        \mu_n^e (\sO_m \mid \sO_m \cup \sE_m) &= \frac{\sum_{\eta \in \cB_{m,n}^e} \lambda^{|\eta_{U_n}|}}{\sum_{\eta \in \cA_{m,n}^e} \lambda^{|\eta_{U_n}|}} = \frac{\sum_{\eta \in \cB_{m,n}^e} \lambda^{|\phi(\eta)_{U_n}|} \cdot \lambda^{|\eta_{U_n}| - |\phi(\eta)_{U_n}|}}{\sum_{\eta \in \cA_{m,n}^e} \lambda^{|\eta_{U_n}|}} \\
        & = \frac{\sum_{\tau \in \cA_{m,n}^e \setminus \cB_{m,n}^e} \sum_{\substack{\eta \in \cB_{m,n}^e \\ \phi(\eta) = \tau}}\lambda^{|\tau_{U_n}|} \cdot \lambda^{|\eta_{U_n}| - |\tau_{U_n}|}}{\sum_{\eta \in \cA_{m,n}^e} \lambda^{|\eta_{U_n}|}} \\
        & = \frac{\sum_{\tau \in \cA_{m,n}^e \setminus \cB_{m,n}^e} \lambda^{|\tau_{U_n}|}  \sum_{\substack{\eta \in \cB_{m,n}^e \\ \phi(\eta) = \tau}} \lambda^{|\eta_{U_n}| - |\tau_{U_n}|} }{\sum_{\eta \in \cA_{m,n}^e} \lambda^{|\eta_{U_n}|}} \le \frac{1}{3}.
    \end{split}
\end{equation*}
Therefore, the mapping $\phi$ involved in~\eqref{eq:shift_bound} plays a key role in the proof. This mapping
will simply be a shift of $\eta$ by one unit in one of the 4 directions, see  Definition~\ref{def:mapping}.
The construction of $\phi$ relies on a geometric description of the interface between regions of different parity, captured through a contour representation, which we will now introduce. 

\subsubsection{The contour}
A sequence of vertices $x_1, x_2, \dots, x_k$ in a graph is called a path if $x_i \sim x_{i+1}$ for all $1 \le i < k$, and the path is called a (simple) cycle if $x_1 = x_k$ and $x_i \ne x_j$ for all $1 \le i < j < k$.
Only such cycles are considered below. 

From now till the end of this subsection, we fix an $m$-odd configuration $\eta \in \cB_{m,n}^e$. 
Consider the subgraph induced by the vertices $\{ i \in \Z^2 \mid \eta_i = \omega^o_i\}$. By the $m$-oddness of $\eta$, there is a connected component of this subgraph, $R$, such that $U_m\subset R$.

To define contours, we construct a graph $\Z^2_{\Diamond}$ as follows: its vertices are the edges of $\Z^2$, with two vertices connected if and only if the corresponding edges are adjacent (share a vertex in $\Z^2$) and are perpendicular.

With any configuration, $\eta$, we can associate a contour in $\Z^2_{\Diamond}$ as described next. Let $\gamma(\eta)$ be the set of edges $(u,v)$ in $\Z^2$ with $u \in R$ and $v \in R^c$, such that there is a path connecting $v$ to infinity that avoids $R$. 
(The last condition ensures that we exclude boundaries caused by ``holes'' in $R$).
Define the contour associated with $\eta$ as the subgraph of $\Z^2_{\Diamond}$ induced by the midpoints of the edges in $\gamma(\eta)$ and denote it by $\Gamma(\eta)$. To ease the notation, we also write $\Gamma, \gamma$ to refer to $\Gamma(\eta), \gamma(\eta)$ if there is no confusion. The construction of the contour is illustrated in Figure~\ref{fig:contour}. 

\begin{figure}[ht]
  \centering
  \begin{subfigure}[t]{0.45\textwidth}
    \includegraphics[width=\linewidth]{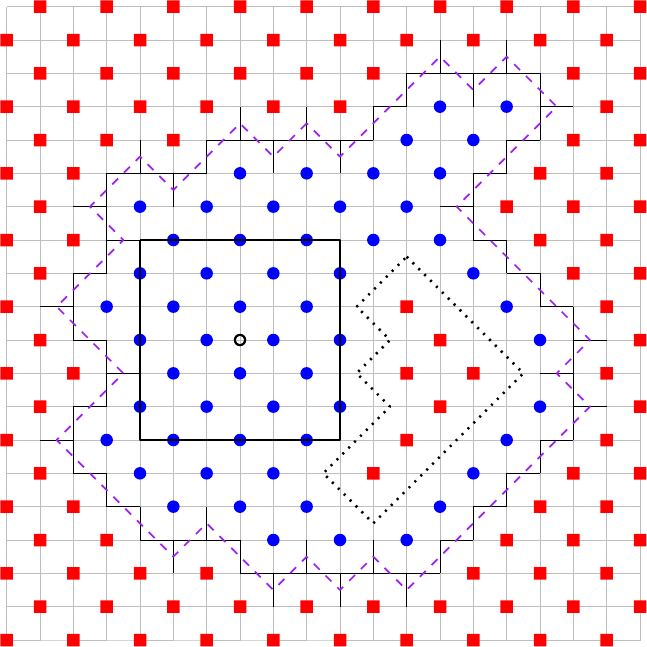}
    \subcaption{Construction of a contour defined by a $3$-odd configuration $\eta$. The odd and even occupied vertices of $\eta$ are shown as \textcolor{blue}{$\bullet$} and \textcolor{red}{$\blacksquare$}, respectively.
   The box $U_3$ is shown as a $6 \times 6$ square centered at the origin (shown as $\circ$).
     The connected component $R$---the subgraph of $\Z^2$ formed by the odd vertices and their neighbors---is surrounded by the dashed cycle in $\Z^2_\Diamond$.
 This cycle forms the contour $\Gamma(\eta)$, which 
  does not include the boundary induced by the hole in $R$, shown as a dotted cycle.}\label{fig:contour}
  \end{subfigure}\hfill
  \begin{subfigure}[t]{0.45\textwidth}
    \includegraphics[width=\linewidth]{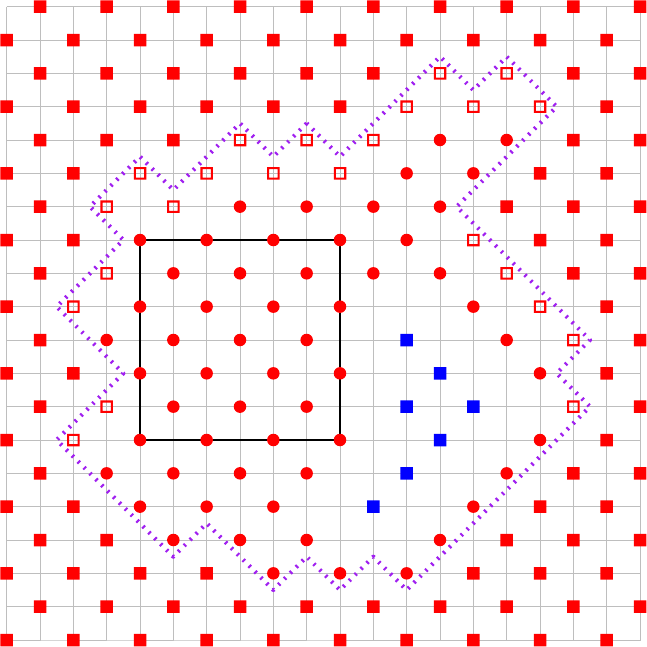}
    \subcaption{{\bf Contour elimination.} This figure shows the result of a shift
 by $s=(0,-1)$ for the configuration $\eta$ in Fig.~\ref{fig:contour}.
 The odd and even occupied vertices are shown in blue and red, respectively.
                     The shifted vertices in $(\supp(\eta) \cap \Gamma_\bullet) + s$ are shown as \textcolor{red}{$\bullet$} and $\textcolor{blue}{\blacksquare}$, and the newly added vertices in the set $I'_s$ are shown as \textcolor{red}{$\square$}. 
         The resulting configuration $\theta_s(\eta)$ is $3$-even (even on $U_3^S$), with the contour surrounding $U_3$ eliminated.}\label{fig:contour-shift}
  \end{subfigure}
  \caption{Illustration of the contour elimination procedure.}\label{fig:contour-combined}
\end{figure}

\begin{proposition}[\cite{blanca2019phase}, Lemmas~3.1--3.3] \label{prop:contour} Let $\eta\in \cB_{m,n}^e$.
    The contour $\Gamma = \Gamma(\eta)$ is a closed cycle in $\Z^2_{\Diamond}$ that separates $R$ from $U_n^c$.
    The length of the contour, $|\Gamma|$, is at least $2 \sqrt{2} m$ and is a multiple of $4$. 
\end{proposition}

Note that a contour, $\Gamma$, has a finite interior $\Gamma_\bullet$ and an unbounded exterior, $\Gamma_\circ$ (by the Jordan curve theorem), and both
subgraphs $\Gamma_\bullet$ and $\Gamma_\circ$ are connected. Also note that $R \subset \Gamma_\bullet$, and $\Gamma_\bullet$ contains the ``holes'' of $R$.

The next lemma is obvious. It will be used repeatedly in the proof.
\begin{lemma}\label{lemma:contour}
    If $(u,v)$ is an edge in $\Z^2$ with $u \in R$ and $v \in R^c$, then $u$ is an even vertex, $v$ is an odd vertex, and $\eta_u = \eta_v = 0$. 
\end{lemma}
 
 \begin{lemma}[\cite{blanca2019phase}, Theorem~2.2, Lemma~3.4]\label{lemma:contour_number} 
    Let $\mathcal{C}_l^m(\eta), \eta\in \cB_{m,n}^e$, be the collection of contours of length $4l$. Then there exists a constant $\alpha_\star$ and a sub-exponentially growing function $g(l)$ such that
    \[
        |\mathcal{C}_l^m| \le g(l) \cdot \alpha_\star^{4l},
    \]
    where $1.55701 < \alpha_\star < 1.58746$.
\end{lemma}
By a subexponentially growing function $g(l)$ we mean a function that satisfies
$g(l)< \exp(\epsilon l)$ for any $\epsilon>0$ and all $l>l_0(\epsilon)$.

\subsubsection{The shift operation}
The $m$-oddness condition guarantees the existence of a contour surrounding $U_m$. 
A standard approach to establish~\eqref{eq:even_boundary} -- the low probability of $m$-odd configurations under $\mu_n^e$ -- is to show that such contours are unlikely. 
To this end, we construct a shift operation $\phi : \cB_{m,n}^e \mapsto \cA_{m,n}^e \setminus \cB_{m,n}^e$ that eliminates the contour enclosing $U_m$. 
We then demonstrate that each image under $\phi$ corresponds to a configuration with significantly higher $\mu_n^e$ probability.

Let $s \in S_{\pm}:=\{(1,0),(-1,0),(0,1),(0,-1)\}$. An \emph{$s$-shift} moves vertex $v$ to $v+s$. The $s$-shift of a subset $S\subset \Z^2$ amounts to moving all of its vertices by $s$, namely, $S+s:= \{v+s \mid v \in S\}$.
The shift of a configuration $\eta \in \cB_{m,n}^e$ in the $s$ direction by definition results in a configuration, $\theta_s(\eta)$, 
obtained as a result of the following two steps (see Fig.~\ref{fig:contour-shift} for an illustration):
\begin{enumerate}
    \item \textbf{Shifting the vertex interior:} Apply the $s$-shift to $\supp(\eta) \cap \Gamma_\bullet$ and leave $\supp(\eta) \cap \Gamma_\circ$ unchanged.
    The resulting configuration has support $I_s := \left( \left( \supp(\eta) \cap \Gamma_\bullet \right) + s \right) \cup \left( \supp(\eta) \cap \Gamma_\circ \right)$.

    \item \textbf{Adding new occupied even vertices:}
    Let $I'_s = \{v \in \Gamma_\bullet \mid v-s \in \Gamma_\circ \}$ be the set of (even) vertices that are moved from $\Gamma_\circ$ to $\Gamma_\bullet$ by applying the $s$-shift. In this step, we make all the vertices in $I'_s$ occupied and let $\theta_s(\eta)$ be the configuration with support $\bar{I}_s := I_s \cup I'_s$.
\end{enumerate}

\begin{lemma}\label{prop:shift}
    Let $\eta \in \cB_{m,n}^e$ and let $s\in S_\pm$.
The shifted configuration $\theta_s(\eta)$ satisfies the following: \\
(a) $\theta_s(\eta)\in \cX_0,$\\ (b)  $|\theta_s(\eta)_{U_n}| - |\eta_{U_n}| = |I'_s|$. \\
Moreover, there exists $s\in S_\pm$ such that\\
(c) $\theta_s(\eta) \in \cA_{m,n}^e \setminus \cB_{m,n}^e$, i.e., it is $m$-even.
\end{lemma}

\begin{proof}
 First let us show (b). 
  To this end we show that $|\eta_{U_n}| = |I_s \cap U_n|$ and $I_s \cap I'_s = \emptyset$. Since $\Gamma_\bullet \subset U_n$, to show that $|I_s \cap U_n| = |\eta_{U_n}|$ it suffices to verify that
(1),  $|( \supp(\eta) \cap \Gamma_\bullet ) + s| = |\supp(\eta) \cap \Gamma_\bullet |$ and (2), $ \supp(\eta) \cap \Gamma_\circ$ is disjoint from $\left( \supp(\eta) \cap \Gamma_\bullet \right) + s$.
    
Both of these follow by definition. Claim (1) follows since shifting by $s$ neither creates nor removes vertices of $\eta_{{U_n}}$ and
    (2) holds since neighbors of every vertex in $\supp(\eta) \cap \Gamma_\bullet$ are also contained in $\Gamma_\bullet$ (Lemma~\ref{lemma:contour}), thus $(\supp(\eta) \cap \Gamma_\bullet) + s \subset \Gamma_\bullet$. 
    
\vspace*{.1in}
Now we will prove that $\theta_s(\eta) \in \cX_0$ for all $s \in S_\pm$. The proof of this part comprises 4 steps. In the first three steps, we show that $\bar I_s := I_s \cup I'_s$ forms an independent set, and in the final step we show that this independent set is maximal.

(1) $I_s \in \sI(\Z^2)$. 
It is clear that $I_s \cap \Gamma_\bullet = (\supp(\eta) \cap \Gamma_\bullet) + s \in \sI(\Z^2)$ and $I_s \cap \Gamma_\circ  = \supp(\eta) \cap \Gamma_\circ \in \sI(\Z^2)$.
Moreover, any two vertices $x\in(\supp(\eta) \cap \Gamma_\bullet) + s \subset \Gamma_\bullet$ and $y\in\supp(\eta) \cap \Gamma_\circ$ are not adjacent in $\Z^2$:
otherwise, $(x,y)$ would be an edge from $R$ to $R^c$ with $\eta_y = 1$. This contradicts Lemma~\ref{lemma:contour}, proving our claim.

(2) $I'_s \in \sI(\Z^2)$. This is straightforward since by Lemma~\ref{lemma:contour}, $I'_s$ is formed of only even vertices.
  
(3) $\bar I_s \in \sI(\Z^2)$. In view of (1)-(2), we only need to show that there are no edges in $\Z^2$ between $I'_s$ and $I_s$.
Let $x \in I'_s$ and $y\in I_s$ and suppose that $x\sim y$. There are two possibilities for $y$, corresponding to the two subsets that form $I_s$.
If $ y \in \supp(\eta) \cap \Gamma_\circ$ then $(x,y)$ is an edge from $R$ to $R^c$ with $\eta_y = 1$, which is ruled out by Lemma~\ref{lemma:contour}.

Now suppose that $y \in ( \supp(\eta) \cap \Gamma_\bullet) + s$ and note that $y \in \Gamma_\bullet$; also $x-s\in R^c$ by definition of $I_s'$.
    
Since we assumed that $x\sim y$, this means that $x=y+s'$ for some $s'\in S_\pm$. 
\begin{enumerate}
    \item[({\em a})] $s'= -s$. Then $x = y-s \in \supp(\eta) \cap \Gamma_\bullet$, which implies $\eta_x=1$. Therefore, $(x-s,x)$ is an edge that connects $R^c$ and $R$, with $\eta_x=1$, which is not possible.
    \item[({\em b})] $s'= s$. Then $y=x-s\in \Gamma_\circ$, but also $y\in \Gamma_\bullet$ by assumption.
    \item[({\em c})] $s'\perp s$ is one of the two remaining directions. In this case, the vertices $y-s \in \supp(\eta) \cap \Gamma_\bullet\subset R$ and $x-s\in R^c$ are adjacent,
    and $\eta_{y-s}=1$, which is again impossible.
\end{enumerate}

  (4) Finally, let us show that $\bar{I}_s  \in \sM(\Z^2)$. 
    Assume toward a contradiction that there exists a vertex $x \notin \bar{I}_s$ such that $\bar{I}_s \cup \{x\} \in \sI(\Z^2)$, or, equivalently, that $(\theta_s(\eta))_{\cross_x} = 0$,  where (by abuse of notation) this refers to a cross-shape with center at $x$ 
    and 0's at all five sites. We will show that in this case there exists an all-zero cross in $\eta$, which contradicts the assumption that $\eta\in\cX_0$.
    
    \begin{enumerate}[label={\upshape(}\alph*{\upshape})]
        \item  First consider the case when $x \in \Gamma_\circ$. 
    If $\cross_x \subset \Gamma_\circ$ then $\theta_s(\eta)_{\cross_x} = \eta_{\cross_x} = 0$, which is impossible.
    If $x\in \Gamma_\circ$ and $\cross_x \cap R \ne \emptyset$, then by Lemma~\ref{lemma:contour} we have $\eta_{\cross_x \cap \Gamma_\bullet} = 0$. 
    At the same time, by definition of $\theta_s$, we have $\theta_s(\eta)_{\cross_x \cap \Gamma_\circ}  =\eta_{\cross_x \cap \Gamma_\circ} =0$. 
This yields the same impossible relation $\eta_{\cross_{x}} = 0$.
   
For the next step, refer to Fig.~\ref{fig:x-neighbors} for the mentioned vertices in the grid and their adjacency 
relations.
    \item It remains to consider the case when $x \in \Gamma_\bullet$. 
    If $\cross_x \subset \Gamma_\bullet$ then $\cross_{x-s} \subset \Gamma_\bullet$ because $x$ cannot be adjacent to any vertex in $I'_s$.
    Therefore, $\theta_s(\eta)_{\cross_x} = \eta_{\cross_{x-s}} = 0$, a contradiction.       
   It remains to consider the case when $\cross_x \cap \Gamma_\circ \neq \emptyset$. We cannot have $x-s \in \Gamma_\circ $ since $x \notin I'_s$.
    Therefore, $x-s \in \Gamma_\bullet$ and $\eta_{x-s} = \theta_s(\eta)_x = 0$. We are going to prove that $\eta_{\cross_{x-s}}=0$, establishing a contradiction.

    \begin{figure}
        \centering
        \includegraphics[width=0.5\linewidth]{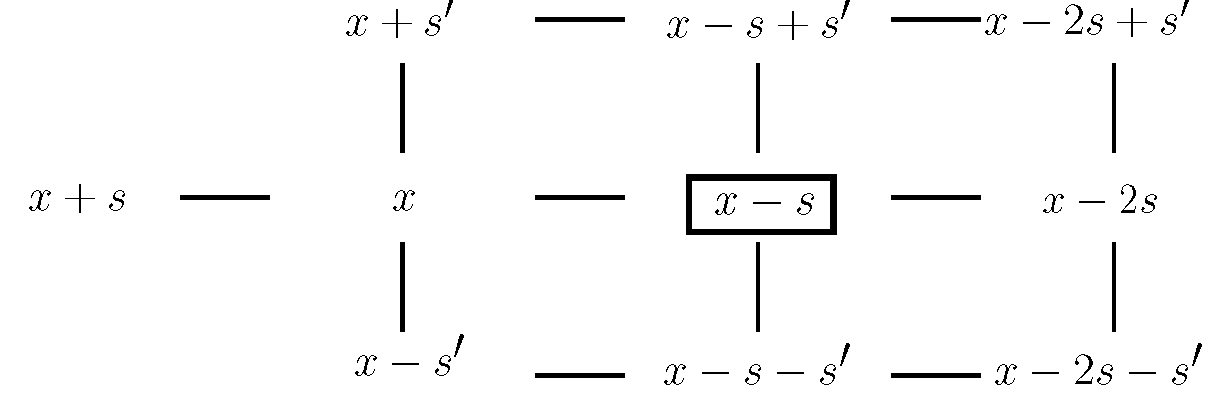}
        \caption{Illustrating Part (4) of the proof of Lemma~\ref{prop:shift}. }
        \label{fig:x-neighbors}
    \end{figure}
    
    \begin{itemize}
    \item $\eta_x = 0$.  To see this, if $x+s \in \Gamma_\circ$ then $\eta_x = 0$ by Lemma~\ref{lemma:contour}.
    On the other hand, if $x+s \in \Gamma_\bullet$ then $\eta_x = \theta_s(\eta)_{x+s} = 0$. 

 \item $\eta_{x-2s} = 0$. Note that $x-2s \in \Gamma_\bullet$ (otherwise $x-s \in I'_s$). Thus, $\eta_{x-2s} = \theta_s(\eta)_{x-s} = 0$.
    
    \item Let $s'\perp s$. We argue that both $\eta_{x-s+s'}$ and $\eta_{x-s-s'}$ are zero,  relying on symmetric arguments. If $x-s+s' \in \Gamma_\circ$ then $(x-s, x-s+s')$ is an edge from $R$ to $R^c$, and $\eta_{x-s+s'} =0$ by Lemma~\ref{lemma:contour}. 
    If $x - s + s' \in \Gamma_\bullet$ then $\eta_{x - s + s'} = \theta_s(\eta)_{x + s'} = 0$.

    Switching to $\eta_{x-s-s'}$, if $x-s-s' \in \Gamma_\circ$ then $(x-s, x-s'-s)$ is an edge from $R$ to $R^c$, and thus $\eta_{x-s'-s} = 0$.
    If $x-s-s' \in \Gamma_\bullet$ then $\eta_{x-s'-s} = \theta_s(\eta)_{x-s'} = 0$.
    \end{itemize}
   \end{enumerate}
Altogether this implies that $\eta_{\cross_{x-s}} = 0$. 

This establishes part (a). 

Part (c). Finally, we show that there exists an $s$ such that $\theta_s(\eta) \in \cA_{m,n}^e \setminus \cB_{m,n}^e$. By the $m$-oddness of $\eta$, $R$ contains at least one of the sets $U_m^\ast$ with $\ast\in\{W,E,S,N\}$.

Assume without loss of generality that $U_m^W \subset R$, then $U_m^W \subset \Gamma_\bullet$, and we have that 
\[
\theta_{(1,0)} (\eta)_{U_m^E} = \eta_{U_m^W} = \omega^o_{U_m^W}.
\]
Therefore, for $i \in {U_m^E}$, we have 
\[
    \theta_{(1,0)} (\eta)_i = \eta_{i - (1,0)} = \begin{cases}
        0 & i \text{ odd} \\
        1 & i \text{ even}
    \end{cases}
\]
since $i - (1,0) \in U_m^W$. Hence, $\theta_{(1,0)} (\eta)_{U_m^E} = \omega^e_{U_m^E}$ and $\theta_{(1,0)}(\eta)$ is $m$-even.
\qedhere
 \end{proof}

Note that both (a) and (b) of this lemma hold for all shifts $s$; it is only part (c), the evenness of the shifted
configuration, where we rely on the assumption that $\eta_i=\omega_i^o$ for $i$ in the row or column added to $U_m$.

\begin{definition}[The mapping $\phi$] \label{def:mapping}
    For $\eta \in \cB_{m,n}^e$ we put $\phi(\eta):=\theta_s(\eta)$, where $s$ be a value of the shift such that $\theta_s(\eta) \in \cA_{m,n}^e \setminus \cB_{m,n}^e$ (such shifts exist by Lemma~\ref{prop:shift}).
\end{definition}

\begin{lemma}\label{lemma:gamma/4}
    Let $\eta \in \cB_{m,n}^e$, we have that
    \begin{equation*}
        |\eta_{U_n}| - |\phi(\eta)_{U_n}| = - \frac{|\gamma(\eta)|}{4}.
    \end{equation*}
\end{lemma}

\begin{proof}
Let $\Gamma_\bullet$ and $\Gamma_\circ$ be the vertex interior and exterior of $\Gamma(\eta)$, respectively. Let $I'_s := \{v \in \Gamma_\bullet \mid v - s \in \Gamma_\circ \}$.
Define a configuration $\omega= \omega^o_{\Gamma_\bullet} \omega^e_{\Gamma_\circ}$. 
Note that $\omega \in \cB_{m,n}^e$ and $\gamma(\omega) = \gamma(\eta)$ by construction. 
Moreover, $\theta_s(\omega) = \omega^e$ for every $s \in S_\pm$. Indeed, $\theta_s(\omega) \in \cX_0$ by the proof of Lemma~\ref{prop:shift} (step (4)), and only even vertices are occupied in $\theta_s(\omega)$.
Therefore, by Lemma~\ref{prop:shift}, we have 
$|\omega^e_{U_n}| - |\omega_{U_n}| = |I'_s|$.

% By Lemma~\ref{prop:shift}, we have that $|\eta_{U_n}| - |\omega^e_{U_n}| = -|I_s'|$.
Let $\gamma_s$ be the set of edges $(u,v)$ with $u \in \Gamma_\bullet, v \in \Gamma_\circ$, and $v = u-s$. 
The map from $\gamma_s$ to $I'_s$ that sends $(u,v)$ to $u$ is bijective, and thus $|\gamma_s| = |I'_s|$.
Moreover, note that $\gamma = \bigcup_{s \in  S_\pm} \gamma_{s}$ and the sets $\gamma_s$ are disjoint. 
Thus, $|\gamma_s| = |\gamma|/4$ for all $s \in S_\pm$, which was to be proved.
\end{proof}

Now we are in a position to prove~\eqref{eq:shift_bound}.
\begin{lemma}\label{lemma:1/3}
    Let $\lambda > \alpha_\star^4$. Fix $n > m$ and let $m$ be sufficiently large and $\tau \in \cA_{m,n}^e \setminus \cB_{m,n}^e$. Then
    \begin{equation*}
        \sum_{\substack{\eta \in \cB_{m,n}^e \\ \phi(\eta) = \tau}} \lambda^{|\eta_{U_n}| - |\tau_{U_n}|} \le \frac{1}{3}.
    \end{equation*}
\end{lemma}

\begin{proof} For a given contour $\Gamma$ there are at most four $\eta \in \cB_{m,n}^e$ for which $\phi(\eta) = \tau$ (at most one for each shift $s \in S_\pm$.) Together with Lemma~\ref{lemma:gamma/4}, we therefore obtain
      \begin{equation*}
        \begin{split}
            \sum_{\substack{\eta \in \cB_{m,n}^e \\ \phi(\eta) = \tau}} \lambda^{|\eta_{U_n}| - |\tau_{U_n}|} = \sum_{\substack{\eta \in \cB_{m,n}^e \\ \phi(\eta) = \tau}} \lambda^{-\frac{|\gamma(\eta)|}{4}}  {\le 4} \sum_{\Gamma \in \mathcal{C}^m}\lambda^{-\frac{|\Gamma|}{4}}.
        \end{split}
    \end{equation*}
    Let $\mathcal{C}^m := \bigcup_{l\ge 4}\mathcal{C}_l^m$ be the collection of all possible contours and
   let us bound the last sum using Lemma~\ref{lemma:contour_number}:
    \begin{equation*}
        {4} \sum_{\Gamma \in \mathcal{C}^m} \lambda^{-\frac{|\Gamma|}{4}} = {4} \sum_{4l \ge 2\sqrt{2} m } \frac{|\mathcal{C}_l^m|}{\lambda^l} \le {4} \sum_{l \ge \frac{\sqrt{2}m}{2}} \frac{g(l)\alpha_\star^{4l}}{\lambda^l} < \frac{1}{3},
    \end{equation*}
    where the last inequality follows since $\alpha_\star^4<\lambda$, so the sum
    on $l$ represents the tail of a convergent series.
\end{proof}

We have established \eqref{eq:even_boundary}. The proof of~\eqref{eq:odd_boundary} proceeds analogously by interchanging the roles of even and odd vertices. In conclusion, \eqref{eq:distinguish} holds true for $\lambda > 6.3506 >\alpha_\star^4$, and this finishes the proof of Theorem~\ref{thm:main}.

\begin{remark}\label{remark: differences} As noted above, the general idea of proving phase coexistence in the hard-core model 
(as well as models with positive temperature) was introduced in Dobrushin's work \cite{dobrushin1968problem}. These proofs find a distinguishing event that has different probability under measures that follow different boundary conditions. In \cite{dobrushin1968problem}, this event was simply occupancy of a single site in $\Lambda\Subset\Z^\nu$. The best known estimate 
of the activity value above which there occurs a phase transition \cite{blanca2019phase} relies on a variation of the $m$-odd and $m$-even events defined above. The versions of the events used in \cite{blanca2019phase} lead to complications in the proof; in particular, issues arise around the proof of their Lemma 7. Choosing a new distinguishing event, we also had to construct a new proof of the main claim, Lemma~\ref{prop:shift}. As one of the steps, we additionally had to establish maximality of the shifted configuration, which is obviously not needed in \cite{blanca2019phase}. 
\hfill$\triangleleft$
\end{remark}

\section{The maximal hard-core model at low activity} \label{sec:low avtivity}
As remarked above, the maximal hard-core model has a peculiar behavior at low activity: since every admissible configuration is an MIS, there are nontrivial ground states. This is unlike the standard hard-core model, where there is only one ground state at low activity, namely, the all-zero configuration. Our main result in this section will be obtained as an application of the Pirogov-Sinai theory, which shows the existence of extremal Gibbs measures generated by the periodic ground states of the maximal hard-core model. A general introduction to this theory is given in numerous places, among which we found \cite[Ch.7]{friedliLattice2018}, \cite{zahradnik1984alternate,zahradnik1998short,cannon2024pirogov,mazel2025high} particularly well attuned to our study. Two main steps have to be cleared before this theory can be applied, namely, classification of the periodic ground states and checking the ``Peierls condition'' on the contour growth as explained below. The bulk of this section is devoted to these two steps.

\subsection{Periodic ground states}\label{sec: periodic ground states}

A configuration $\omega \in \cX_0$ is called periodic if there exist linearly independent $u, v \in \Z^2$ such that $\omega_i = \omega_{i +  u} = \omega_{i+v}$ for all $i \in \Z^2$.
Let $\cX_0^{\text{\rm per}}$ be the set of periodic configurations in $\cX_0$. 
For $\omega \in \cX_0^{\text{\rm per}}$, the following limit exists and is called the \emph{energy density}:
\[
e(\omega) := \lim_{n \to \infty} \frac{1}{|U_n|} \sH_{U_n}(\omega) = \lim_{n \to \infty} \frac{|\omega_{U_n}|}{|U_n|} (- \ln \lambda),
\]
where $|\omega_\Lambda|$ denotes the number of occupied sites in $\Lambda\Subset\Z^2$.
%We also occasionally use the notation $\supp(\omega_\Lambda):=\{i\in \Lambda\mid\omega_i=1\}$ for the {\em support} of $\omega$ within $\Lambda$.

It is known~\cite[Lemma 7.4]{friedliLattice2018} that $\eta \in \cX_0^{\text{\rm per}}$ is a (periodic) ground state if and only if its energy density is minimal, that is 
\[
e(\eta) = \inf_{\omega \in \cX_0^{\text{\rm per}}} e(\omega).
\]

For $\omega \in \cX_0^{\text{\rm per}}$, define the density of $\omega$ as follows:
\[
\delta(\omega) := \lim_{n \to \infty} \frac{|\omega_{U_n}|}{|U_n|}.
\]
Note that when $\lambda > 1$, $e(\omega)$ is minimized by maximizing $\delta(\omega)$. On the other hand, when $0 < \lambda < 1$, $e(\omega)$ is minimized by minimizing $\delta(\omega)$. 
Therefore, periodic ground states are equivalently defined as the sparsest and densest periodic configurations of $\cX_0$ 
depending on whether the activity $\lambda<1$ or $\lambda > 1$.

\begin{theorem}\label{prop: density}
    For $\omega \in \cX_0^{\text{\rm per}}$, we have $\frac{1}{5} \le \delta(\omega) \le \frac{1}{2}$. 
    There are exactly $2$ periodic configurations with density $\frac{1}{2}$ and exactly $10$ periodic configurations with density $\frac{1}{5}$.
\end{theorem}
\begin{proof} 
Both statements about the lower bound are elementary once one realizes that we seek tilings of $\Z^2$ with ``unit'' crosses\footnote{In the language of coding theory, the crosses are unit balls in the $\ell_1$ metric, and tilings correspond to single-error-correcting perfect codes in this metric.}.
To make the connection, let $\omega \in \cX_0^{\text{\rm per}}$ be a periodic configuration, 
defined by the condition $\omega_i  = \omega_{i + v}$ for all $i \in \Z^2$ and some $v = (v_1, v_2)$. This means that there is a rectangular tile with sides $v_1,v_2$ whose shifts fill the plane, and then 
   $$
   \delta(\omega) = \frac{|\omega_{\Z_{v_1 \times v_2}}|}{v_1 v_2}.
   $$
The first step is to use periodicity to identify this tile with the torus $\torus_{v_1,v_2}=\{(i,j) \mid i \in \Z\, (\text {mod\,} v_1), j \in \Z\, (\text {mod\,} v_2)\}$, so that every $0$ in $\omega_{\Z_{v_1 \times v_2}}$ is adjacent to at least one $1$. Next let us count
the number, $N$, of pairs $(u,v)$ in which $u$ and $v$ are adjacent sites in $\torus_{v_1,v_2}$, 
with the property that $\omega_u = 1$ and $\omega_v = 0$. Since every occupied site $u \in \supp(\omega_{\Z_{v_1 \times v_2}})$ is adjacent to 
four unoccupied sites, we have $N = 4  |\omega_{\Z_{v_1 \times v_2}}|$. On the other hand, every unoccupied site $v$ in the torus 
$\torus_{v_1,v_2}$ is adjacent to at least one and at most four occupied sites; so we also have $v_1v_2 - |\omega_{\Z_{v_1 \times 
v_2}}| \le N \le 4 (v_1v_2 - |\omega_{\Z_{v_1 \times v_2}}|)$. Substituting $N$, we obtain
     $$
v_1v_2 - |\omega_{\Z_{v_1 \times v_2}}| \le 4  |\omega_{\Z_{v_1 \times v_2}}| \le 4  (v_1v_2 - |\omega_{\Z_{v_1 \times v_2}}|),
     $$
from which we obtain $\frac15 \le \delta(\omega) \le \frac12$. Moreover, we have $\delta(\omega) = \frac15$ iff equality holds in the first inequality above, i.e., every 
unoccupied site $v$ in the torus $\torus_{v_1,v_2}$ is adjacent to exactly one occupied site. Similarly, we have $\delta(\omega) = \frac12$ iff equality holds in the second inequality, i.e., every unoccupied site $v$ in the torus $\torus_{v_1,v_2}$ is adjacent to exactly four occupied sites. Focusing on the $\delta(\omega) = \frac12$ case first, we see that this holds for a periodic configuration $\omega$ iff every $1$ in $\omega$ is adjacent to four $0$s, and every $0$ is adjacent to four $1$s. It follows that $\omega$ can only be one of the two configurations $\omega^e$ and $\omega^o$. 

For $\delta(\omega) = \frac15$ to hold for a periodic configuration $\omega$, every $1$ must be adjacent to four $0$s, and every $0$ must be adjacent to a single $1$. 
In this case, the crosses centered on occupied sites form a tiling, namely (see Figures~\ref{fig:sparst-states-shifts-a}, \ref{fig:sparse-mirror}):
   $$
    \torus_{v_1,v_2}=\bigsqcup\limits_{i\in \supp(\omega_{\Z_{v_1\times v_2}})} \!\!\cross_i\;.
  $$
Periodic tilings of $\Z^2$ with unit crosses are well known. Let us check that there are exactly 10 distinct tilings of this kind, which form a single equivalence class under $\Z^2$ isometries.
%\footnote{Note that a pair of tilings that can be mapped on each other by a global isometry are still different because they define different Gibbs measures.}.  
W.l.o.g. set $\omega_{(0,0)}= 1$, placing a cross $\cross_{(0,0)}$ at the origin. By symmetry, there are two ways to place another cross to form a tiling, namely with center at $(2,1)$ or $(1,2)$. Say we choose $(2,1)$, then by inspection, this forces crosses $\cross_{(1,3)}$ and
$\cross_{(-1,2)}$,
and this pattern extends periodically; see Fig.~\ref{fig:sparst-states-shifts-a}. Overall this tiling forms a sublattice $\cL_1$ with the basis $(2,1),(-1,2)$ and fundamental volume 5. Choosing $(1,2)$ in the first step leads to its reflection with respect to the vertical axis, i.e., sublattice $\cL_2$
with the basis $(1,2), (-2,1)$. These two lattices/tilings are shown in Fig.~\ref{fig:sparse-mirror}. The remaining 8 configurations arise as $\Z^2$-shifts of these sublattices. For instance, in
the case of $\cL_1$, the shifts are by each of the vectors $(0,1),(0,2),(1,1)$, and $(1,2)$, as illustrated in  Fig.~\ref{fig:sparst-states-shifts-b}.
\end{proof}

 \begin{figure}[ht]
    \centering
    \subcaptionbox{A periodic configuration of density $\frac15$. 
    The crosses, with centers shown as $\bullet$, constitute a tiling of the discrete torus 
    $\torus_{5\times 5}$. Together with its reflection, Fig.~\ref{fig:sparse-mirror}, 
    these tilings yield the 10 ground states of the system. \label{fig:sparst-states-shifts-a}}[0.48\textwidth]
    {\includegraphics[width=0.25\textwidth]{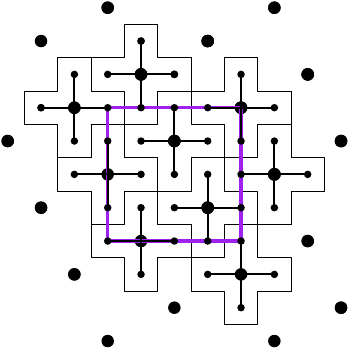}}
   \hspace*{0in}
    \subcaptionbox{Sublattice $\cL_1$ with the origin shown as $\textcolor{red}{\bullet}$ 
    and other lattice points as $\bullet$. Shifting the lattice by $(0,2)$, we obtain 
    another ground state with occupied sites shown as $\circ$. 
    The other 3 states in this group are obtained as shifts of $\cL_1$ by 
    $(1,1),(0,1)$, and $(1,2)$.  \label{fig:sparst-states-shifts-b}}[0.48\textwidth]
    {\includegraphics[width=0.25\textwidth]{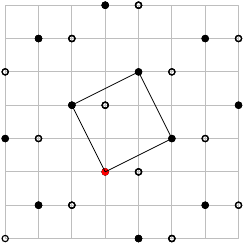}}
      \caption{Construction of periodic ground states.}
    \label{fig:sparse-proof}
\end{figure}
\vspace*{-0.1in}
\begin{figure}[ht]
    \centering
    \includegraphics[width=0.45\linewidth]{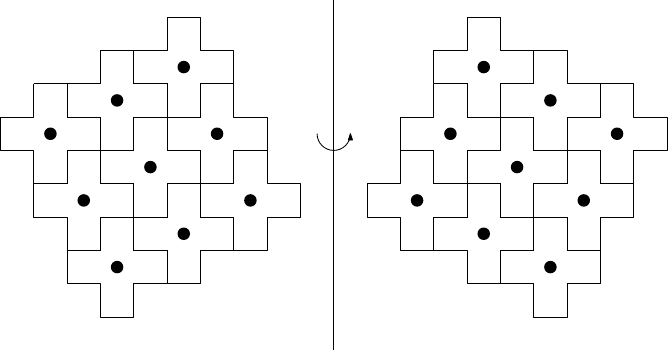}
    \caption{Lattice tilings $\cL_1$ and $\cL_2$.}
    \label{fig:sparse-mirror}
\end{figure}

\subsection{The Peierls condition \texorpdfstring{($\lambda < 1$)}{lambda<1}}
We have shown that when $\lambda < 1$, there are exactly 10 periodic ground states. 
One naturally expects that for sufficiently small $\lambda$, typical configurations are local perturbations of these ground states. 
The Pirogov–Sinai Theory asserts precisely that when $\lambda$ is small enough, all extremal Gibbs measures arise as thermodynamic limits of finite-volume Gibbs measures with boundary conditions given by (stable) periodic ground states. 
Moreover, configurations sampled from these extremal measures typically differ from the corresponding ground state only by local perturbations.

To apply the Pirogov–Sinai theory, we must verify the Peierls condition. 
Loosely speaking, this condition controls the increase of energy caused by a perturbation of a ground state, and is satisfied if this increase is proportional to the size of the perturbation. This condition can be rigorously expressed using the contour formalism.

Let $\cQ$ be the set of all $5 \times 5$ patterns (tiles) that occur in exactly recoverable configurations. 
We may then reinterpret configurations in $\cX_0$ using a block-spin representation: partition $\mathbb{Z}^2$ into disjoint $5 \times 5$ blocks of the form $\mathbb{Z}_{5 \times 5} + 5t$, where $t \in \Z^2$, and encode each block by an element of $\cQ$.

Let $\mathcal{G}$ denote the set of 10 periodic ground states for $\lambda < 1$, and let $\cQ_{\text p}$ be the set of $5 \times 5$ patterns arising from periodic ground states. Then, in this block-spin representation, each ground state corresponds to a constant configuration $Q^{\mathbb{Z}^2}$ for some $Q \in \cQ_{\text p}$.

Let $\omega \in \cX_0$ and $\bar{\omega}$ be the corresponding block-spin representation.
Denote by $\cQ(l,m) := \Z_{5 \times 5} + 5(l,m)$ the $5 \times 5$ box centered at $(5l,5m)$, and let $\bar{\omega}_{l,m}$ be the restriction of $\omega$ to $\cQ(l,m)$.
Let $Q \in \cQ_{\text p}$.
We say that the block $\cQ(l,m)$ is \emph{$Q$-correct} in $\omega$ if $\bar{\omega}_{l + i,m + j} = Q$ for all $i, j \in \{0,\pm 1\}$. 
In other words, a correct block at $(l,m)$  matches the pattern $Q$ together with all of its 8 neighboring blocks.
We call a block {\em incorrect} if it is not $Q$-correct for any $Q \in \cQ_{\text p}$. 

We define a {\em contour} in $\omega$ as a maximal connected component $\Gamma$ of incorrect blocks.
The energy associated with a contour $\Gamma$ is determined by the number of occupied vertices that $\Gamma$ encloses. 
The Hamiltonian restricted to $\Gamma$ is
   $$
    \cH_{\Gamma}(\omega) = - |\omega_\Gamma| \ln \lambda,
   $$
where $|\omega_\Gamma|$ denotes the number of occupied sites within $\Gamma$.
Let $\eta$ be one of the periodic ground states. The relative Hamiltonian of $\omega$ with respect to $\eta$, restricted to $\Gamma$, is given by 
   $
    \cH_\Gamma(\omega \mid \eta) =  (|\eta_\Gamma| - |\omega_\Gamma|) \ln \lambda.
   $
By definition of the ground state, $\cH_\Gamma(\omega \mid \eta)\ge 0$. The Peierls condition, moreover, requires that
   $$
\cH_\Gamma(\omega \mid \eta) \ge \rho |\Gamma|
   $$
for some constant $\rho > 0$. Using the expression for $\cH$ and recalling that $\lambda<1$, we can write this explicitly as
\begin{equation}\label{eq:peierls}
    |\omega_\Gamma| - |\eta_\Gamma| \ge \alpha |\Gamma|,
\end{equation}
where $\alpha := - \frac{\rho}{\ln \lambda} > 0$.

\begin{remark}
    There is a technical issue concerning the way of verifying the Peierls condition. Often $\Gamma$ is defined as the ``thick'' boundary of all the incorrect blocks within $\cQ(k,m)$, i.e., the union of all contours in the
    region, see e.g., \cite[Ch.7]{friedliLattice2018}. Another possibility appearing in the literature  (\cite[Sec.1.5]{zahradnik1998short}, \cite{mazel2019high}) is verifying condition \eqref{eq:peierls} for individual contours, and then aggregating these inequalities for the final result. We will follow the second option without further mention. \hfill$\triangleleft$
\end{remark}

One approach to analyzing the number of occupied vertices in configurations from  $\cX_0$ is through the \emph{Delaunay triangulation} defined by their support. Let $\omega \in \cX_0$.
To every occupied site, $i$,  there corresponds a polygon formed of the points of $\R^2$
that are closer to $i$ than to any other occupied site of $\omega$. The collection of these
polygons forms a Voronoi partition of $\R^2$ defined by $\omega$. The Delaunay triangulation defined by $\omega$ is a triangulation of $\reals^2$ constructed as the dual graph of the Voronoi partition defined by $\omega$; namely, the circumcenters of the triangles are the vertices of the Voronoi diagram. 

Note that for a given $\omega$, there may exist several different triangulations; for instance, in Fig.~\ref{fig: Delaunay}, the vertices can be grouped into triples in more than one way. At the same time, they all share a defining property, namely, the open circumcircles of the triangles do not contain any occupied sites of $\omega$ other than the vertices of their triangles.

If $\omega$ is one of the periodic ground states, then
every triangle has side lengths $\sqrt{5}, \sqrt{5}, \sqrt{10}$ and area $\frac{5}{2}$, as is evident from 
Fig.~\ref{fig: Delaunay}.

\begin{figure}[ht!]
\begin{center}
\includegraphics[width=.4\linewidth]{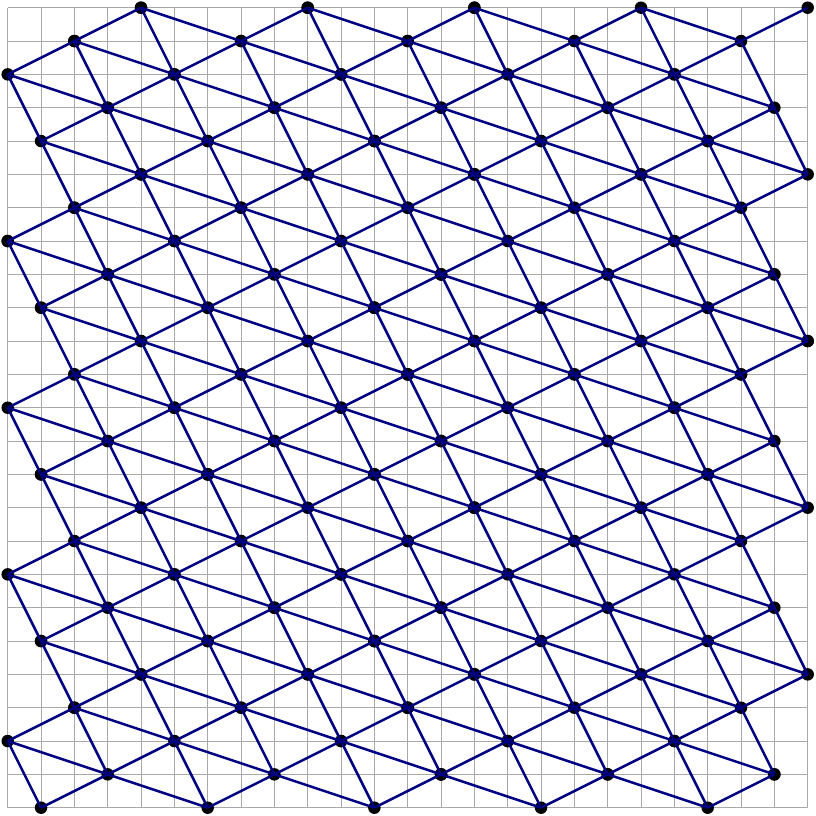}
\caption{A Delaunay triangulation for the ground state $\cL_1$.}\label{fig: Delaunay}
\end{center}
\end{figure}

Call a triangle {\em regular} if its side lengths are $\{\sqrt{5}, \sqrt{5}, \sqrt{10}\}$ and call it 
{\em defective} otherwise.
An incorrect block is said to be t-defective if it intersects at least one defective triangle (even on a single vertex). 
An incorrect block is called n-defective if it is not t-defective: these blocks are neighbors of a t-defective block, which themselves take values from $\cQ_{\text p}$. 

The following lemma is proved by examining all possible triangles in a Delaunay triangulation of an MIS. 
\begin{lemma}\label{prop: 5/2}
   For $\omega\in\cX_0$, the area of any defective triangle in the Delaunay triangulation defined by $\omega$ is at most 2.
\end{lemma}
\begin{proof}
We will make use of the following two claims:
\begin{itemize}
    \item The area of a triangle with integer vertices is an integer or a half-integer.
  \item Any circle in $\R^2$ of radius $\bar\rho:=\sqrt\frac 52\approx1.581$ contains a point in $\omega$.
\end{itemize} 
The first claim follows by Pick's theorem, which asserts that the area equals the number of interior points plus half the number of boundary points minus 1.
%    $
%\#(\text{interior points})+$ $\frac12\#(\text{boundary points})-1.
%  $  
For the second claim, let $z\in\R^2$. The nearest integer point, $p$, is no further than $\sqrt 2/2$. This point is either in $\omega$ or at distance 1 to an occupied vertex $i\in\omega$. Notice that the latter case gives the maximum value of $\rho(\omega)$ --- the covering radius of the configuration. Indeed, applying the cosine law for the triangle $(z,p,i)$, we find that $\rho(\omega)\le \bar\rho$.
For instance, this is the value of $\rho$ for $\omega$ in Fig.~\ref{fig:sparst-states-shifts-a}.

We will examine all the possible triangles with vertices in $\omega$. By shift invariance of an MIS, it suffices to examine all the possible triangles with one vertex at $(0,0)$, since if there is an MIS with a triangle of some shape elsewhere, its shift is also an MIS.
Therefore, all the triangles will be assumed to have one of their vertices at the origin. 

Let us note two useful properties of the circumcircles. First, observe that a triangulation cannot include a triangle with circumradius greater than $\bar\rho$. Indeed, in this case, the circumcircle of the triangle would have to contain points of
$\omega$ distinct from the vertices of the triangle, contradicting the defining property of the triangulation. Next, the circumradius of a triangle with positive area is at least half the length of each side. 

Suppose the triangle contains a pair of vertices, $(0,0)$ and $(x,y)$, at $\ell_1$ distance $5$ (that one of the vertices is taken as $(0,0)$ is again justified by shift invariance). It suffices to consider triangles with the vertices in a single quadrant, say the first one. By symmetry, the only possibilities are $(x,y)\in\{(0,5),(1,4),(2,3)\}$, resulting in each case in a side of the triangle of length $> 2\bar\rho$, so that the circumradius is $> \bar\rho$. Therefore, these
triangles cannot appear in the triangulation. The same argument also rules out triangles with vertices at $\ell_1$ distances greater than 5.

Now assume that the $\ell_1$ distance between $(0,0)$ and $(x,y)$ is 4, then it suffices to assume that $(x,y)\in \{(0,4),(1,3),(2,2)\}$.
The first option is ruled out since $2\bar\rho< 4$, so we are left to examine the other two options. If $(x,y)=(1,3)$ and the
third vertex is $(u,v)$, then the area of the triangle is $\frac 12|v-3u|$, where $u+v\le 4$ and $u\ge 0,v\ge 0$. Moreover, the circumradius of the triangle equals\footnote{{The circumradius of a triangle with side lengths $a, b, c$ and area $A$ is    $
   \frac{abc}{4A}$.}}
   $$\sqrt{\frac{5}{2}} \cdot \frac{\sqrt{(u^2+v^2)((u-1)^2+(v-3)^2)}}{|v-3u|}.
   $$
This yields the
following table of possible points below, which also records the area and the ratio of the circumradius to $\bar{\rho}$ for each of the resulting triangles:

\vspace*{.1in}   \begin{tabular}{ccccccccccccc}
$u$     &0&0&0&1&2&3&1&2&3&1&2\\
$v$     &1&2&3&0&0&0&1&1&1&2&2\\
Area    &0.5&1&1.5&1.5&3&4.5&1&2.5&4&1&2 \\
$\text{Circumradius} / \bar{\rho}$ & $\sqrt{5}$ & $\sqrt{2}$ & $1$ & $1$ & $\frac{\sqrt{10}}{3}$ & $\frac{\sqrt{13}}{3}$ & $\sqrt{2}$ & $1$ & $\frac{\sqrt{5}}{2}$ & $\sqrt{5}$ & $1$ \\
   \end{tabular}
   
\medskip
\noindent The triangles with circumradius $> \bar{\rho} \ge \rho(\omega)$ are infeasible, i.e., cannot arise in the Delaunay triangulation. This rules out all triangles with area $> 2.5$. The triangle with area 2.5 is regular.

If $(x,y)=(2,2)$, then the area of the triangle is $|v-u|$ with $u+v\le 4$ and $u,v\ge 0$, and the 2 possible options $(u,v)\in\{(0,3), (3,0)\}$ that yield area $\geq 2.5$, give rise to infeasible triangles even though their circumradius is $\bar{\rho}$. For instance,
suppose $(u,v)=(3,0)$, then the circumcenter of the triangle $T=(0,0),(2,2),(3,0)$ is $(\frac 32,\frac12)$. The open circumcircle of $T$ contains
$(1,0),(2,0),(1,1),(2,1)$ which are therefore unoccupied. Moreover, $(0,1)$ and $(1,2)$ are adjacent to the vertices of $T$ and hence also vacant, and thus $(1,1)$ has no occupied neighbors, contradicting maximality.

The case $x+y=3$ is exhausted similarly. There are two possible options, $(x,y)\in\{(0,3),(1,2)\}$, and the area 
equals $\frac 12|xv-yu|$. By inspection, all feasible triangles have area at most 2.

Finally, the case of $\ell_1$ distance 2 is addressed by a direct inspection of all triangles. The proof is complete.
 % This creates a finite search space, where we go over all MISs in a region and examine whether the circumcircle of the triangles contains other occupied sites. The search finds no triangles of area more than 5/2, and the only triangles of that area are regular ones. The next possible area is 2 since the area of a triangle with integer vertices is an integer or a half integer by Pick's theorem.
\end{proof}

\vspace*{.1in} 
Let $\omega\in\Omega$ and let $\torus_k$ be a $5k\times 5k$ torus in $\Z^2$ (for the moment, it does not matter that the sides are multiples of 5). Denote by $\sT_\omega(\torus_k)$ a Delaunay triangulation of $\torus_k$ defined by $\omega$.
\begin{lemma} For any $\omega\in \cX_0$, 
\begin{enumerate}
    \item the triangles in $\sT_\omega(\torus_k)$ form a tessellation of $\torus_k$;
    \item  the number of triangles $|\sT_\omega(\torus_k)|=2|\omega_{\torus_k}|$.
 \end{enumerate}
 \end{lemma}
\begin{proof}
    The first part follows by definition. The second claim is justified as follows. Place a disk of area 1 around each occupied site; since the distance between the sites is at least $\sqrt 2$, the disks do not overlap. Each triangle intersects exactly three such disks---one for each of its vertices. Moreover, since the sum of the interior angles is $\pi,$ the sum of the areas of the three sectors cut out by the triangle is $1/2$. Rephrasing, the area within the disks covered by each triangle is exactly 1/2. 
    %\Nnote{ OK, I see this now -- sum of interior angles in a triangle is $\pi/2$ -- but it would still be useful to have a line of explanation.} 
    On the other hand, every disk is fully covered by the triangles, so the count of triangles is twice the number of disks, which is also the number of occupied sites.
\end{proof}

The specific version of \eqref{eq:peierls} that we establish appears next.
\begin{theorem}\label{thm: Peierls} Let $\omega\in \cX_0$, let $\eta$ be a periodic ground state, and let $\Gamma$ be the contour defined above. Then
  $$
  |\omega_\Gamma|-|\eta_\Gamma|\ge \max\Big(1, \frac1{270}|\Gamma|\Big).
  $$
\end{theorem}
\begin{proof}
Let $\omega \in \cX_0$ and let $\Gamma$ be a contour formed of incorrect blocks. 
Observe that the blocks in the immediate neighborhood of $\Gamma$ are $Q$-correct for some values of $Q \in \cQ_{\text p}$.
Recall that the $Q$-constant block-spin configuration is a periodic ground state. Therefore, $\omega_\Gamma$ can be extended using $Q$-tiles, which gives us a configuration  $\tau$ on $\torus_k$ for a  sufficiently large $k$. 
In other words, we define a configuration $\tau$ on $\torus_k$ such that
$\tau_\Gamma=\omega_\Gamma$ and $\tau_{\Gamma^c}=\eta'_{\Gamma^c}$ for some other periodic ground 
state of our model\footnote{If $\Gamma$ is not
simply connected, it may happen that $\omega_{\Gamma^c}$ agrees with different periodic ground states in different connected
components of $\torus_k$ defined by $\Gamma$. This will not matter for the argument below since we rely on the number of occupied sites rather than
on the ground states themselves.}. Importantly, $|\tau_{\Gamma^c}|=|\eta_{\Gamma^c}|$, where $\Gamma^c=\torus_k\backslash\Gamma$.
     By Lemma~\ref{prop: 5/2}, any defective triangle has area at most 2. If we replace one or more triangles of area 5/2 with defective (smaller) ones, then to tile $\torus_k$ we will use more triangles than the number of triangles in 
 $\sT_\eta(\torus_k)$ (this number equals 
$\frac{|\torus_k|}{2.5}=10k^2$, where $|\cdot|$ denotes the area).  
The total number of triangles in a triangulation is even (it is twice the number of occupied sites), so the total count of triangles increases by at least $2$, and the number of occupied sites increases by at least $1$. 

The size of $\sT_\tau(\torus_k)$ can be estimated as follows. Denote by $R$ a generic triangle and let 
$N_0(\tau)$ and $N_d(\tau)$ be the number of regular and defective triangles in $\sT_\tau(\torus_k)$, respectively. We have
   \begin{align*}
   |\torus_k|&=\sum_{\text{defective } R} |R| + 2.5 N_0(\tau) \\
   & \le \, 2 N_d(\tau) + 2.5 N_0(\tau)
   \end{align*}
using the fact that the area of a defective triangle is at most 2, proved in Lemma~\ref{prop: 5/2}. Re-arranging this inequality, we obtain
   $$
    N_0(\tau) + N_d(\tau) \ge \frac{|\torus_k|}{2.5} + \frac15 N_d(\tau).
   $$
Thus, the triangulation $\sT_\tau$ increases the count of triangles over $\sT_\eta$, namely, 
    $$
    |\sT_\tau(\torus_k)|\ge |\sT_\eta(\torus_k)|+ \frac15 N_d(\tau).
    $$
Since each $t$-defective incorrect block is surrounded by at most $8$ $n$-defective incorrect blocks, all of which enter
$\Gamma,$ and each defective triangle
intersects at most 3 incorrect blocks, we obtain
   $$
    3 N_d(\tau) \ge \sharp\{\text{$t$-defective blocks}\} \ge \frac{|\Gamma|}{9}.
   $$
Note that $\omega_\Gamma=\tau_\Gamma$ (since the only changes are outside
$\Gamma$) and $|\omega_{\Gamma^c}|\ge |\tau_{\Gamma^c}|$ (since $\omega_{\Gamma^c}$ may span defective triangles, and
$\tau_{\Gamma^c}$ gives rise only to regular ones). We continue as follows:
  \begin{align*}
      |\omega_{\torus_k}|&=|\omega_\Gamma|+|\omega_{\Gamma^c}|\ge |\tau_\Gamma|+|\tau_{\Gamma^c}|\\
      &=\frac12|\sT_\tau(\torus_k)|\ge\frac12\Big(|\sT_\eta(\torus_k)|+\frac{N_d(\tau)}5\Big)\\
      &\ge \frac12\Big(|\sT_\eta(\torus_k)|+\frac{|\Gamma|}{135}\Big)\\
      &=|\eta_{\torus_k}|+\frac{|\Gamma|}{270}. 
  \end{align*}
Since we argued earlier that $|\omega_{\torus_k}|\ge |\eta_{\torus_k}|+1$, the proof is complete.
\end{proof}

The results obtained so far, namely, the enumeration of periodic ground states (Theorem~\ref{prop: density}) and verification of the Peierls condition, prepare the ground for the use of the Pirogov-Sinai theory, which implies several claims about extremal Gibbs measures and related concepts. They are summarized in the next theorem.
%%%%%
\begin{theorem} The maximal hard-core model at low activity $\lambda <1$ has 10 periodic ground states, obtained from each other by global $\Z^2$ isometries. For sufficiently small $\lambda$, each ground state $\eta$ gives rise to an extremal periodic 
Gibbs measure, $\mu_\eta$, and all such measures are generated by the periodic ground states.
The measures $\mu_\eta$ for different periodic ground states are mutually singular. For $\lambda\to0$ and any $\eta$, the measure
$\mu_\eta$ converges weakly to $\delta_\eta$.
\end{theorem}
%%%%%
The first claim of this theorem was proved in Theorem~\ref{prop: density} above; thus, in the maximal hard-core case at low activity, periodic ground states form a single equivalence class. 
    The remaining parts of the proof amount to collecting the relevant results from the literature devoted to the
Pirogov-Sinai theory, among which we single out \cite{zahradnik1984alternate,zahradnik1998short},
as well as Ch.~7 of \cite{friedliLattice2018}. Details of the application of this theory for the (standard) hard-core case appear in \cite[Theorem III]{mazel2025high}, which also gives specific references to the literature required to complete the proof. Here we limit ourselves to a brief summary.\\
\indent It is known that every extremal Gibbs measure $\mu_\eta$ is generated by a periodic ground state, $\eta$, namely,
        $$
        \mu_\eta=\lim_{\Lambda\Uparrow \Z^2}\mu_\Lambda(\cdot|\eta),
        $$
where the limit is in the van Hove sense. This follows from 
\cite{zahradnik1984alternate} and the main result of \cite{dobrushin1985problem}. Generally, there may exist periodic ground states that do not give rise to (extremal) Gibbs measures; however, there is at least one state, $\eta$,  that does \cite{zahradnik1984alternate}. Moreover, all ground states $\eta'$ obtained from $\eta$ by $\Z^2$ isometries also give rise to Gibbs measures $\mu_{\eta'}$. 
The remaining claims also follow immediately from \cite{zahradnik1984alternate}.

\section{Concluding remarks}

The approach of this paper applies to other recoverable systems, defined by similar local recovery rules that 
can be described by defining a potential $\Phi_B$, where $B$ is the recovery region. 
For instance, in a companion paper \cite{barg2026recoverable}, we have studied the maximal hard-core model on the triangular lattice $\mathbb{A}^2$, with recovery region defined by the immediate neighbors. A follow-up work \cite{mazel2026:triangular} extended the maximal hard-core model on $\mathbb{A}^2$, proving phase coexistence for high- and low-activity regimes.

Here we outline several research directions related to the interplay between interaction models and
recoverable systems. First, information storage in the system under positive temperature deserves a detailed study. There is a natural tradeoff between the integrity of the stored data and
the storage capacity: for $\beta<\infty$ the Gibbs measure assigns positive
probability to configurations containing defects, i.e., sites at which the
recovery rule fails, and lowering $\beta$ increases the entropy of the
measure (the storage rate) together with the density of defects (the error
rate). Quantifying this tradeoff, e.g., as the curve
$\beta\mapsto(\delta(\beta),\psi(\beta))$, where $\delta(\beta)$ is the
defect density and $\psi(\beta)$ the pressure per site, is an open problem
whose solution would locate the achievable rate--reliability region of the storage model.

Another open problem left by this work is the uniqueness question of the Gibbs measure defined by the potential~\eqref{eq:potential} beyond the high-temperature case. A possible approach to the proof could rely on the Dobrushin-Shlosman phase uniqueness criterion \cite{dobrushin1985constructive,radulescu1987dobrushin} or a connection with disagreement percolation \cite{van1994disagreement}.

Switching to the hard-core model, we know that the Gibbs distribution is not unique for both small and large $\lambda$. At present, it is not even clear whether there are any values of $\lambda$ for which this distribution is unique. Of particular interest is the case $\lambda=1$, for which the equilibrium states are precisely the measures of maximal entropy of the subshift of MISs, so that uniqueness of the Gibbs measure is
equivalent to uniqueness of the measure of maximal entropy. We believe that uniqueness holds at $\lambda=1$, but we are presently unable to devise an approach for proving this.

Yet another research direction is the relationship between the activity parameter $\lambda$ and the density
of occupied sites ($1$’s) in typical configurations. Since $\lambda$ controls the tendency of a vertex to be occupied, it directly influences the system’s density. For sufficiently large $\lambda$, the model concentrates on configurations close to $\omega^e$ or $\omega^o$, both of density 1/2. In contrast, the sparsest recoverable configurations, which are dominant when $\lambda$ is sufficiently small, have density 1/5. Since $\psi$ is convex in $\log\lambda$ \cite[Ch.~3,4]{friedliLattice2018}, the density
$\rho(\lambda)=\lambda\,\partial_\lambda\psi(\lambda)$ is nondecreasing wherever it is defined, interpolating between the density $1/5$ of the sparsest recoverable configurations and the density $1/2$ of $\omega^e,\omega^o$.
Experimental observations suggest that, at $\lambda = 1$, the density of typical configurations is approximately 1/3. This raises a natural question: can one rigorously establish a density curve as a function of $\lambda$, and in particular, determine whether the density is 1/3 at $\lambda = 1$ exactly or approximately. Apart from being of mathematical interest, the resolution of this question would have implications for the design of storage systems, and we leave this as an open problem.

\vspace*{.1in}
{\sc Acknowledgments.} We are grateful to Ron Peled and Senya Shlosman for helpful discussions of 
this research. This project began while the second author was visiting the Indian Institute of Science in Bangalore in August 2024, and he is grateful to the Institute for making this visit possible and for the welcoming atmosphere. The first two authors were partially supported by the US NSF grants CCF-2330909 and CCF-2526035.

\appendix%\label{Appendix}

\section{Relation to subshifts of finite type}\label{app: SFT}
Here we give an independent proof of the fact that the system $\cX_0$ has positive entropy, based on its description as a subshift of finite
type (SFT) and a general fact about such subshifts. An SFT is a collection of configurations $\cX\subset \Omega$ defined by forbidding finitely many local patterns. Clearly, $\cX_0$ is an SFT defined by forbidding all-zero crosses and pairs
of adjacent ones. 

An SFT $\cX$ is called \emph{strongly irreducible} if there exists a constant $r >0$ such that for all finite regions $A, B \Subset \Z^2$ with $\ell_1$ distance $d_1(A,B)\ge r$, for every pair of configurations $\omega, \eta \in \cX$, there exists $\sigma \in \cX$ such that $\sigma_A = \omega_A$ and $\sigma_B = \eta_B$. In words, if $\omega_A$ and $\eta_B$ are extendable configurations in $\cX$, then they are also jointly extendable provided that $A$ and $B$ are sufficiently separated.
 
\begin{proposition}
    The SFT $\cX_0$ is strongly irreducible.
\end{proposition}
\begin{proof}
    Let $A, B \Subset \Z^2$ be a pair of finite regions with $d_1(A,B)\ge 4$. 
    Let $\bar A = A \cup \partial A$ and $\bar B = B \cup \partial B$, where $\partial A$, for instance,
    refers to the boundary of the set $A$.
    For any $\omega, \eta \in \cX_0$, we will construct a $\sigma \in \cX_0$ such that $\sigma_A  = \omega_A$ and $\sigma_B = \eta_B$ by the following procedure.

    \begin{enumerate}
        \item Initialize $\sigma = 0^{\Z^2}$. Fix a sufficiently large $n$ such that both $A$ and $B$ are contained in the box $U_n = \{-n, \dots, n\}^2$. 
        \item Let $\sigma_{\bar A} = \omega_{\bar A}$, and $\sigma_{\bar B} = \eta_{\bar B}$. For all $i \in U_{n+2}^c$, let $\sigma_i = 1$ if $i$ is even, otherwise let $\sigma_i = 0$.
        \item Let $I = \supp (\sigma_{U_{n+4}})$ be the set of sites $i$ such that $\sigma_i=1$. 
        Pick $i \in  U_{n+2} \setminus (A \cup B)$ such that $ \{i\} \cup I$ forms an independent set in $\Z^2$, and let $\sigma_i = 1$.
        \item Repeat the previous step until $I$ becomes an MIS in $U_{n+4}$.
    \end{enumerate}
    
    Note that upon completing step 2), the support of $\sigma$ is an independent set in $\Z^2$, the support of $\sigma_{U_{n+{4}}^c}$ 
    is an MIS, and every $0$ in $\sigma_A\cup\sigma_B$ is adjacent to some $1$. Steps 3) and 4) are applied to a finite region, eventually making (the support of) $\sigma$ an MIS. Therefore, by Lemma~\ref{lemma:MIS}, $\sigma \in \cX_0$. 
\end{proof}

It was shown in~\cite{burton1994nonuniqueness} that strongly irreducible SFTs with at least $2$ configurations have positive topological entropy (see~\eqref{eq:topo_entropy}). 
Moreover, a translation-invariant measure $\mu$ on a strongly irreducible SFT has maximal entropy (which equals the topological entropy) if and only if for each $n$, the conditional distribution on $U_n$ under the boundary condition $\omega$ is the uniform distribution among all those configurations $\tau$ such that $\tau_{U_n} \omega_{U_n^c}$ form an element of the SFT (uniform conditional probabilities). The maximal hard-core Gibbs measure has uniform conditional probabilities when $\lambda = 1$, thus
it has maximal entropy, and forms a maxentropic measure for the corresponding SFT. Generally, such measures for SFTs form 
the subject of a long line of research in the literature; see \cite{burton1994nonuniqueness}, \cite{haggstrom1996phase}, \cite{meyerovitch2013gibbs}, \cite{garcia2019extender} and their references.

\section{Proof of Theorem~\ref{thm:fast-mixing}}\label{app:coupling}
   This proof is inspired by \cite[Theorem~5.8]{MarkovMixing}.
Recall that a coupling of two Markov chains with transition matrix $\bfP$ is the process 
$(X_t, Y_t)_{t=0}^\infty$ such that  both $(X_t)_{t=0}^\infty$ and $(Y_t)_{t=0}^\infty$ are Markov chains with transition matrix $P$ (but possibly different initial distributions). 
Any coupling of two Markov chains $(X_t)$ and $(Y_t)$ defined on the same state space $\cS$ can be modified to have the property that if $X_s = Y_s$, then $X_t = Y_t$ for $t \ge s$.

\begin{proposition}~{\rm\cite[Corollary~5.5]{MarkovMixing}} \label{thm:mixing_up}
    Suppose that for each pair of states $x,y \in \cS$ there is a coupling $(X_t,Y_t)$ with $X_0 = x$ and $Y_0 = y$. 
    Let $\tau_c := \min\{ t: X_s = Y_s~\text{for all } s \ge t\}$ be the coalescence time. 
    Then
    \begin{equation*}
        d(t) \le \max_{x,y \in \cS} \bfP_{x,y} \left(\tau_c > t\right).
    \end{equation*}
\end{proposition}

    We construct a collection of random variables $\{X_t^x \mid x \in \Omega_\Lambda, t = 0, 1, 2, \dots\}$ such that for each $x \in \Omega_\Lambda$, the sequence $(X_t^x)_{t=1}^\infty$ is the Glauber chain starting with $x$. 
    We generate an i.i.d. sequence $Z_1, Z_2, \dots$  such that $Z_t = (w, U)$,  where $w \in \Lambda$ is a uniform random vertex and $U$ is a uniform random variable on $[0,1]$.
    Let $X_0^x = x$.
    At time $t$, for each chain $(X_t^x)_{t=1}^\infty$, we update $X_{t-1}^x$ by updating site $w$ according to $\pi_w(\cdot \vert X_{t-1}^x)$. 
    In other words, let $X_{t-1}^x(w)$ be the value stored on vertex $w$ and $\sigma = X_{t-1}^x \vert_{\Lambda \setminus \{w\}}$ be the configuration restricted outside $w$.
    \begin{equation*}
        X_t^x(w) = \begin{cases}
            0 & \text{if } U \le \pi_w(0 \vert \sigma) \\
            1 & \text{if } U > \pi_w(0 \vert \sigma),
        \end{cases}
    \end{equation*}
    and $X_t^x (v) = X_{t-1}^x (v)$ for $v \ne w$.

    For $x,y \in \Omega_\Lambda$, let $\rho(x,y) = \sum_{i \in \Lambda} \1_{x_i \ne y_i}$ be the Hamming distance, which is a valid metric on $\Omega_\Lambda$ (states of the Glauber chain).
    Suppose $\rho(x,y) = 1$ and denote by $v \in \Lambda$ such that $x_v \ne y_v$.
    We consider the distance after the one-step update, $\rho(X_1^x, X_1^y)$.

    The distance is zero if and only if $v$ is selected for updating, which happens with probability $1/n$. Hence,
     \begin{equation*}
        {\mathbb P}\left(\rho(X_1^x, X_1^y) = 0\right) = \frac{1}{n}.
    \end{equation*}

    Suppose that vertex $w \ne v$ is selected for updating. 
    If $\| w - v \|_1 > 2$ then $\pi_w(0 | x) = \pi_w(0 | y)$ and $\rho(X_1^x, X_1^y) = 1$.
    If $ 0 < \| w - v \|_1 \le 2$, the probability that $x_w, y_w$ are updated to different values is 
    $|\pi_w(0 | x) - \pi_w(0 | y)|$, and thus,
    \begin{equation*}
        \begin{split}
            {\mathbb P} \left(\rho(X_1^x, X_1^y) = 2\right) &\le \frac{12}{n} \left(\sup_{w: 0 < \|w-v\|_1 \le 2} \left |\pi_w(0 \vert x) - \pi_w(0\vert y) \right | \right) \\
                &\le \frac{12}{n} \left(\frac{2}{1+e^{-10\beta}} -1\right),
        \end{split}
    \end{equation*}
    where the last inequality follows from the observation that $ 1 - \frac{1}{1 + e^{-10\beta}}\le \pi_w(0 | \omega) \le \frac{1}{1 + e^{-10\beta}}$ for any $\omega \in \Omega_\Lambda$.
    Indeed, $\pi_w(0 | \omega) \le \frac{e^{ 5 \beta}}{e^{-5\beta} + e^{5\beta}}$ with equality when $w$ is surrounded by $1$'s:
    \[\begin{matrix}
        & & 0 & & \\
        & 0 & 1 & 0 & \\
        0 & 1 & w & 1 & 0 \\
         & 0 & 1 & 0 & \\
         & & 0 & & 
    \end{matrix}.\]
    Similarly, $\pi_w(1 |\omega) \le \frac{e^{5\beta}}{e^{5\beta} + e^{-5 \beta}}$ with equality when $w$ is surrounded by $0$'s:
    \[
    \begin{matrix}
        & & 0 & & \\
        & 0 & 0 & 0 & \\
        0 & 0 & w & 0 & 0 \\
         & 0 & 0 & 0 & \\
         & & 0 & & 
    \end{matrix}
    \]
       Therefore, 
    \begin{equation}\label{eq:Erho_dist1}
        \begin{split}
            \E \left [ \rho(X_1^x, X_1^y) -1 \right ] &= -1 \cdot {\mathbb P}\left(\rho(X_1^x, X_1^y) = 0\right) + 1 \cdot {\mathbb P}\left(\rho(X_1^x, X_1^y) = 2\right) \\
            &\le \frac{1}{n} \left(\frac{24}{1+e^{-10\beta}} -13\right) = -\frac{c_\beta}{n}.
        \end{split}
    \end{equation}

    Now suppose that $\rho(x,y) = r$. 
    There are configurations $x_0 = x, x_1, x_2, \dots, x_r = y$ such that $\rho(x_k,x_{k-1}) = 1$ for $k = 1,\dots, r$, and by the triangle inequality and linearity of expectation,
    \begin{equation*}
        \begin{split}
            \E \left[ \rho(X_1^x, X_1^y) \right] \le \sum_{k=1}^r \E \left[ \rho(X_1^{x_k},X_1^{x_{k-1}}) \right] 
            \le r \left(1 - \frac{c_\beta}{n}\right) = \rho(x,y) \left(1 - \frac{c_\beta}{n}\right).
        \end{split}
    \end{equation*}

    Notice that the random vector $(X_t^x, X_t^y)$ conditioned on the event that $X_{t-1}^x = x_{t-1}$ and $X_{t-1}^y = y_{t-1}$ has the same distribution as $(X_1^{x_{t-1}}, X_1^{y_{t-1}})$.
    Hence, by~\eqref{eq:Erho_dist1},
    \begin{equation*}
        % \begin{split}
            \E \left[ \rho(X_t^x, X_t^y) \mid X_{t-1}^x = x_{t-1}, X_{t-1}^y = y_{t-1} \right] = \E \left[ \rho(X_1^{x_{t-1}}, X_1^{y_{t-1}}) \right] 
            \le \rho (x_{t-1},y_{t-1}) \left(1 - \frac{c_\beta}{n}\right).
        % \end{split}
    \end{equation*}

    Taking an expectation over $(X_{t-1}^x, X_{t-1}^y)$ on both sides, we have 
    \begin{equation}\label{eq:Einductive}
        \E \left[ \rho(X_t^x, X_t^y) \right] \le \E \left[ \rho (X_{t-1}^x,X_{t-1}^y) \right] \left(1 - \frac{c_\beta}{n}\right). 
    \end{equation}
    Iterating~\eqref{eq:Einductive} shows that 
    \begin{equation*}
        \E \left[ \rho(X_t^x, X_t^y)  \right] \le \rho(x,y) \left(1 - \frac{c_\beta}{n}\right)^t \le n e^{- \frac{c_\beta}{n}t}.
    \end{equation*}

    Now consider the coalescence time, $\tau_c$, of the two chains $(X_t^x)$ and $(X_t^y)$. Note that ${\mathbb P}(\tau_c > t) = {\mathbb P}(X_t^x \ne X_t^y) = {\mathbb P} \left(\rho(X_t^x, X_t^y) \ge 1\right)$, so that by Markov's inequality,
    \begin{align*}
        {\mathbb P}(\tau_c > t) &= {\mathbb P} \left(\rho(X_t^x, X_t^y) \ge 1\right) \le \E \left[ \rho(X_t^x, X_t^y)  \right] \\ &\le \rho(x,y) \left(1 - \frac{c_\beta}{n}\right)^t \le n e^{- \frac{c_\beta}{n}t}.
    \end{align*}
    By Proposition~\ref{thm:mixing_up}, $d(t) \le n e^{-t \frac{c_\beta}{n}}$, so if $t > \frac{n}{c_\beta} \left(\ln n + \ln 1/\epsilon\right)$ then $d(t) \le \epsilon$.
The proof is complete.

\section{Existence of the maximal hard-core Gibbs measure}\label{App: existence of measure}

Our proof follows the Dobrushin-Lanford-Ruelle (DLR) approach to the construction. Let $\Lambda \Subset \Z^2$. 
For each $\omega \in \Omega$, we define 
\begin{equation}\label{eq:prob_kernel}
    \pi_\Lambda(\{\tau_\Lambda \omega_{\Lambda^c}\} \mid \omega) := \pi_{\Lambda}(\tau_\Lambda \mid \omega) = \begin{cases}
    \frac{\lambda^{|\tau_\Lambda|}}{Z_\Lambda^\omega} & \text{ if } \tau_\Lambda \omega_{\Lambda^c} \in \cX_0 \\
    0 & \text{ otherwise}
\end{cases}
\end{equation}
where $Z_\Lambda^{\omega_{\Lambda^c}} = \sum_{\tau_\Lambda: \tau_\Lambda \omega_{\Lambda^c} \in \cX_0} \lambda^{|\tau_\Lambda|}$, and $|\tau_\Lambda|$ is the number of $1$'s contained in $\tau_\Lambda$. 
Note that~\eqref{eq:prob_kernel} defines a probability measure on $(\cX_0, \cF_0)$ if $\omega \in \cX_0$ (where $\cF_0$ is the corresponding cylinder $\sigma$-algebra).
These probability measures are consistent in the following sense: for every fixed $\Delta \subset \Lambda \Subset \Z^2$ and $\sA \in \cF$, we have
\begin{align*}
        \pi_{\Lambda} \pi_{\Delta} (\sA \mid \omega) & := \sum_{\tau_\Lambda} \pi_\Lambda(\tau_\Lambda \mid \omega) \pi_\Delta(\sA \mid \tau_\Lambda \omega_{\Lambda^c}) \\
& = \sum_{\substack{\tau_\Lambda \\ \tau_\Lambda \omega_{\Lambda^c} \in \cX_0}} \pi_\Lambda (\tau_\Lambda \mid \omega)   
  \sum_{\substack{\eta_\Delta \\ \eta_\Delta \tau_{\Lambda \setminus \Delta} \omega_{\Lambda^c} \in \cX_0}} \1_{\sA} (\eta_\Delta 
  \tau_{\Lambda \setminus \Delta} \omega_{\Lambda^c}) \pi_{\Delta} (\eta_\Delta \mid \tau_\Lambda \omega_{\Lambda^c}) \\
& = \sum_{\substack{\tau_\Lambda \\ \tau_\Lambda \omega_{\Lambda^c} \in \cX_0}} \sum_{\substack{\eta_\Delta \\ \eta_\Delta \tau_{\Lambda   
    \setminus \Delta} \omega_{\Lambda^c} \in \cX_0}}  \frac{\lambda^{|\tau_\Lambda|}}{Z_{\Lambda}^{\omega_{\Lambda^c}}} \1_{\sA} (\eta_\Delta 
     \tau_{\Lambda \setminus \Delta} \omega_{\Lambda^c}) \frac{\lambda^{|\eta_\Delta|}}{Z_\Delta^{\tau_{\Lambda\setminus \Delta} 
 \omega_{\Lambda^c}}}. \\
& = \sum_{\tau_{\Lambda \setminus \Delta}}
   \sum_{\substack{\tau_\Delta \\ \tau_\Delta \tau_{\Lambda \setminus \Delta} \omega_{\Lambda^c} \in \cX_0}} 
    \sum_{\substack{\eta_\Delta \\ \eta_\Delta \tau_{\Lambda \setminus \Delta} \omega_{\Lambda^c} \in \cX_0}}
    \frac{\lambda^{|\tau_\Delta| + |\tau_{\Lambda \setminus \Delta}|}}{Z_\Lambda^{\omega_{\Lambda^c}}} \1_{\sA} (\eta_\Delta \tau_{\Lambda \setminus \Delta} \omega_{\Lambda^c})\frac{\lambda^{|\eta_\Delta|}}{Z_\Delta^{\tau_{\Lambda\setminus \Delta} \omega_{\Lambda^c}}} \\
\intertext{(here and on the next line the sum on $\tau_{\Lambda \setminus \Delta}$ runs over all patterns such that there is a
$\tau'_\Delta \in \Omega_\Delta$ for which $\tau'_\Delta \tau_{\Lambda 
   \setminus \Delta} \omega_{\Lambda^c} \in \cX_0$)}
& =  \sum_{\tau_{\Lambda \setminus \Delta}} \sum_{\substack{\eta_\Delta \\ \eta_\Delta \tau_{\Lambda \setminus \Delta} \omega_{\Lambda^c} \in 
   \cX_0}} \frac{\lambda^{|\eta_\Delta|+|\tau_{\Lambda \setminus \Delta}|}}{Z_\Lambda^{\omega_{\Lambda^c}}} \1_{\sA}(\eta_\Delta \tau_{\Lambda \setminus \Delta} \omega_{\Lambda^c})\sum_{\substack{\tau_\Delta \\ \tau_\Delta \tau_{\Lambda \setminus \Delta} \omega_{\Lambda^c} \in \cX_0}} \frac{\lambda^{|\tau_\Delta|}}{Z_\Delta^{\tau_{\Lambda \setminus \Delta} \omega_{\Lambda^c}}} \\
& = \pi_{\Lambda}(\sA \mid \omega).
\end{align*}
The last equality above holds because the sum on $\tau_\Delta$ on the previous line equals 1 by definition.
Therefore, the family $\{\pi_{\Lambda}\}_{\Lambda \Subset \Z^2}$ is a \emph{specification}, and for all
$\Lambda\subset U_n$, we have
\begin{equation}\label{eq:consistencyMIS}
    \pi_{U_n} = \pi_{U_n} \pi_\Lambda.
\end{equation}
It can be shown that the set of probability measures on $(\cX_0, \cF_0)$ is sequentially compact~\cite[Theorem~6.24]{friedliLattice2018}.
Therefore, there exists a subsequence $({n_k})_{k \ge 1}$ such that $\pi_{U_{n_k}} \Rightarrow  \mu$ when $k \to \infty$, where $\Rightarrow$ denotes weak convergence. To show the probability measure specified by~\eqref{eq:prob_kernel} exists, we need to verify that $\mu  = \mu \pi_\Lambda$ for all $\Lambda \Subset \Z^2$. 

Denote by $\mu(f)$ the expectation of a function $f: \cX_0 \to \reals$ with respect to the probability measure $\mu$. We will 
limit ourselves to continuous functions $f$, which in our context means the following. We say that a sequence of configurations $\omega^{(n)} \in \cX_0$ converges to $\omega \in \cX_0$ (written as $\omega^{(n)} \to \omega$) if for all $N$, there exists an $n_0$ such that 
     $$
\omega^{(n)}_{U_N} = \omega_{U_N} \quad \forall n \ge n_0.
     $$
A function $f: \cX_0 \mapsto \R$ is said to be continuous if $\omega^{(n)} \to \omega $ implies $ f(\omega^{(n)}) \to f(\omega)$.
If $f$ is continuous, then it is also bounded and $\cF_0$-measurable. %\cite[Exercises~6.10, 6.11]{friedliLattice2018}.

We need the following lemma to proceed.
\begin{lemma}[\cite{friedliLattice2018}, Lemma~6.22]
    Let $\mu, \nu$ be probability measures defined on $(\cX_0, \cF_0)$. 
    We have $\mu = \nu$ if and only if $\mu(f) = \nu(f)$ for all \emph{continuous} functions $f : \cX_0 \mapsto \R$. 
\end{lemma} 

For a probability kernel $\{\pi_\Lambda\}_{\Lambda \Subset \Z^2}$ and a continuous function $f$, define
     $$
\pi_\Lambda f(\omega) := \sum_{\eta_\Lambda \in \Omega_\Lambda}  f(\eta_\Lambda \omega_{\Lambda^c}) \pi_\Lambda (\eta_\Lambda \mid \omega),
     $$
and note that $\pi_\Lambda f$ is again a continuous (thus bounded and measurable) function~\cite[Exercise 6.13]{friedliLattice2018}. 
Moreover, it can be shown~\cite[Exercise~6.6]{friedliLattice2018} that $\nu \pi_\Lambda (f) = \nu (\pi_\Lambda f)$ for every probability measure $\nu$ on $(\cX_0, \cF_0)$.

To ease the notation, we write $\mu_n := \pi_{U_n}(\cdot \mid \omega)$, recalling that there exists a subsequence $(\mu_{n_k})_{k \ge 1}$ such that  $\mu_{n_k} \Rightarrow \mu$ as $k \to \infty$. 
Now for every continuous function $f : \cX_0 \mapsto \R$, we have 
   \begin{align*}
    \mu \pi_\Lambda (f) &= \mu(\pi_\Lambda f) = \lim_{k \to \infty} \mu_{n_k} (\pi_\Lambda f)\\
    &= \lim_{k \to \infty} \mu_{n_k} \pi_\Lambda (f) = \lim_{k 
    \to \infty} \mu_{n_k}(f) = \mu(f),
    \end{align*}
where the second-to-last equality follows by the consistency relation~\eqref{eq:consistencyMIS}.

\section{Proof of Theorem~\ref{thm:H_0 bound}} \label{App:zero-temp-entropy}

\subsection{Lower bound}

Fix $m,n\ge 1$, and let $\cG^0_{m\times n}$ denote the set of
configurations whose supports form MISs in
the subgraph induced by $\Z_{m\times n}$.
For each $k\ge 1$, place independently chosen
configurations from $\cG^0_{m\times n}$ in the $k^2$ blocks
\[
(a(m+1),b(n+1))+\Z_{m \times n}, \qquad 0\le a,b<k
\] 
in the top-left corner of each block.
These blocks are separated by one-site gaps, so their occupied sites
form an independent set, which can be extended to an MIS on $\Z^2$. 
Every unoccupied site in each block already has an occupied
neighbor within that block. Hence the extension preserves all selected
block configurations, and distinct choices give distinct compatible
patterns on $\Z_{k(n+1) \times k(m+1)}$.
Consequently, distinct choices give distinct restrictions to
$\Z_{k(m+1)\times k(n+1)}$, and hence
\[
|\cG_{k(m+1)\times k(n+1)}|
\ge |\cG^0_{m\times n}|^{k^2}.
\]
Taking logarithms, dividing by the number of sites, and letting
$k\to\infty$ with $m,n$ fixed, we obtain
\begin{equation}\label{eq:H_0-low}
H_0 = \lim_{k\to\infty}
\frac{\log_2|\cG_{k(m+1)\times k(n+1)}|}
     {k^2(m+1)(n+1)} \ge
\frac{\log_2|\cG^0_{m\times n}|}{(m+1)(n+1)}.
\end{equation}

% Abusing notation, let $\cG^0_{{m \times n}}$ be the set of configurations on $\Z_{m \times n}$ whose support forms an MIS on $\Z_{m \times n}$. 
Note that $|\cG^0_{m \times n}|$ is the number of MISs on $\Z_{m \times n}$, which has already been computed by Oh~\cite{Oh2017maxindset}.
It was shown that $H_0\ge 0.29$.
Oh's method is computationally efficient but expensive in memory. 
To derive a better estimation, we rely on a standard transfer matrix method.

{\bfit Transfer matrix:} Let $\mathbf{S}(m) \subset \Omega_{\Z_{m \times 2}}$ be the set of binary configurations such that there are no adjacent $1$'s.
It is convenient to think of $\mathbf{S}(m)$ as an $m\times 2$ binary matrix.
We call $\mathbf{S}(m)$ the state space, and $s = (s_L, s_R) \in \mathbf{S} (m)$ a left (right) state if every $0$ in the left (right) column $s_L$ (resp. $s_R$) is adjacent to at least one $1$.
Let $s^i = (s^i_L, s^i_R), s^j = (s^j_L, s^j_R)$ be two states in $\mathbf{S}(m)$. 
There is a transition from state $s^i$ to state $s^j$ if and only if $s^i_R = s^j_L$ and for the configuration $(s^i_L, s^i_R, s^j_R) \in \Omega_{\Z_{m \times 3}}$, every $0$ in the middle column $s^i_R$ is adjacent to at least one $1$, and there is no adjacent $1$'s in the configuration $(s^i_L, s^i_R, s^j_R)$.
After the transition, a new column is added $s^j_R$ to the configuration.
The transfer matrix $T_m$ is an $|\mathbf{S}(m)| \times |\mathbf{S}(m)|$ binary matrix with entries $T_m({s^i, s^j}) = 1$ if there is a transition from $s^i$ to $s^j$ and $0$ otherwise.

% To evaluate {$|\cG^0_{{m\times n}}|$} we construct configurations 
Configurations from $\cG^0_{m \times n}$ can be constructed by state transitions as shown in Example~\ref{example:transfer}.
To construct a configuration from $\cG^0_{m \times n}$, we need to start from a left state and end with a right state. 
Therefore, $|\cG^0_{m \times n}| = \mathbf{1}_L \cdot T_m^{n-2} \cdot \mathbf{1}_R^T$, where $\mathbf{1}_L$ ($\mathbf{1}_R$) is the $|\mathbf{S}(m)|$ dimensional row vector that takes value $1$ at all left (right) states and $0$ otherwise. 
The lower bound in the theorem is obtained by taking $(m,n)=(12,132)$ and finding $|\cG^0_{m \times n}|$ by computer.
We obtain
\(
|\cG^0_{12 \times 132}| \ge5.9329\times 10^{156},
\)
and thus,
\(
H_0 \ge 0.3012. 
\)

\begin{example}\label{example:transfer}
    When $m = 2$, there are 7 states in $\mathbf{S}(2)$:
\begin{gather*}
    s^0 = \begin{bmatrix}
        0 & 0 \\
        0 & 0
    \end{bmatrix},
    s^1 = \begin{bmatrix}
        0 & 0 \\
        0 & 1
    \end{bmatrix},
    s^2 = \begin{bmatrix}
        0 & 1 \\
        0 & 0
    \end{bmatrix},
    s^3 = \begin{bmatrix}
        0 & 0 \\
        1 & 0
    \end{bmatrix}, \\
    s^4 = \begin{bmatrix}
        0 & 1 \\
        1 & 0
    \end{bmatrix},
    s^5 = \begin{bmatrix}
        1 & 0 \\
        0 & 0
    \end{bmatrix},
    s^6 = \begin{bmatrix}
        1 & 0 \\
        0 & 1
    \end{bmatrix}.
\end{gather*}

The left states are $s^3, s^4, s^5, s^6$, the right states are $s^1, s^2, s^4, s^6$.
The transfer matrix takes the form
    \begin{equation*}
        \begin{bmatrix}
0 & 0 & 0 & 0 & 0 & 0 & 0 \\
0 & 0 & 0 & 1 & 1 & 0 & 0 \\
0 & 0 & 0 & 0 & 0 & 1 & 1 \\
0 & 0 & 1 & 0 & 0 & 0 & 0 \\
0 & 0 & 0 & 0 & 0 & 1 & 1 \\
0 & 1 & 0 & 0 & 0 & 0 & 0 \\
0 & 0 & 0 & 1 & 1 & 0 & 0 \\
\end{bmatrix}.
    \end{equation*}
    The sequence of transitions 
    \[
    s^6 \to s^3 \to s^2 \to s^6 \to s^4 \to s^6
    \]
    gives a configuration in $\cG^0_{2\times 7}$:
    \begin{equation}\label{config}
        \begin{bmatrix}
            1 & 0 & 0 & 1 & 0 & 1 & 0 \\
            0 & 1 & 0 & 0 & 1 & 0 & 1
        \end{bmatrix}.
    \end{equation}
\end{example}

\subsection{Upper bound}

Let $F(m,n) \subseteq \Omega_{\Z_{m \times n}}$ be the set formed of the configurations such that no two $1$'s are adjacent and no $0$ is adjacent to four\footnote{It is possible for $0$'s on the boundary to have no neighbors that take value $1$.} $0$'s.
Since $\cG_{m \times n} \subset F(m,n)$, we have
  \begin{equation}\label{eq:H_0-up}
    H_0 = \lim_{m,n \to \infty} \frac{\log_2|\cG_{m \times n}|}{mn} \le \lim_{m,n \to \infty} \frac{\log_2|F(m,n)|}{mn},
  \end{equation}
  where the second limit exists due to subadditivity.

We now apply the following results of~\cite{forchhammer2000} to derive an upper bound of~\eqref{eq:H_0-up}.
A pair of configurations in $F(m,n)$, $n \ge 4$ with identical rightmost and leftmost two columns is called a two-seam cylinder configuration of circumference $2p = 2(n-2)$.
Denote such configurations by $\overline{F}(m,2p)$. 
Let $\overline{F}(2, 2p)$ be the state space. 
There is a transition from state $s^i = \begin{pmatrix}
    s^i_T \\
    s^i_B
\end{pmatrix}$ to $s^j = \begin{pmatrix}
    s^j_T \\
    s^j_B
\end{pmatrix}$  
if the bottom row of $s^i$ is identical to the top row of $s^j$, i.e., $s^i_B = s^j_T$, and $\begin{pmatrix}
    s^i_T \\
    s^i_B \\
    s^j_B
\end{pmatrix} \in \overline{F}(3,2p)$ (see Example~\ref{example:seam}).
Let $\overline{T}_{2p}$ be the corresponding transfer matrix and $\overline{\lambda}_{2p}$ be the largest eigenvalue.
The capacity of the two-seam cylinder of circumference $2p$ is given by the limit
\begin{equation} \label{eq:seam-bound}
\overline{H}(2p):= \lim_{m \to \infty} \frac{\log_2 |\overline{F}(m,2p)|}{2pm} = \frac{\log_2 \overline{\lambda}_{2p}}{2p}.
\end{equation}
The second equality again follows from~\cite[Theorem~4.4.4]{Lind_Marcus_2021}.

\begin{example}\label{example:seam}
    Here are two-seam configurations (states) of circumference $2p = 6$ and width $m=2$, and there is a transition from~\eqref{eq:two_seam1} to~\eqref{eq:two_seam2}.
  \vspace*{.1in}  \begin{equation}\label{eq:two_seam1}
        \begin{bmatrix}
            0 & 0 & 0 & 0 & 0  \\
            0 & 1 & 0 & 1 & 0 
        \end{bmatrix},
        \begin{bmatrix}
            0 & 0 & 1 & 0 & 0  \\
            0 & 1 & 0 & 1 & 0
        \end{bmatrix},
    \end{equation}
    \begin{equation}\label{eq:two_seam2}
        \begin{bmatrix}
            0 & 1 & 0 & 1 & 0 \\
            1 & 0 & 0 & 0 & 1
        \end{bmatrix},
        \begin{bmatrix}
            0 & 1 & 0 & 1 & 0 \\
            1 & 0 & 1 & 0 & 1
        \end{bmatrix}.
    \end{equation}
\end{example}

\vspace*{.1in}\begin{theorem}~\cite[Theorem~1]{forchhammer2000}
    The double limit satisfies
    \begin{equation}\label{eq:seam_upper_bound}
        \lim_{m,n \to \infty} \frac{\log_2 |F(m,n)|}{mn} \le  \overline{H}(2p),
    \end{equation}
    where $p$ is an integer at least $2$.
\end{theorem}

Setting $p = 7$ in~\eqref{eq:seam_upper_bound} and combining with~\eqref{eq:seam-bound}, we obtain $H_0 \le \frac{\log_2\overline{\lambda}_{14}}{14} =  0.34085026$.

\end{document}